\documentclass[11pt]{amsart}

\usepackage[T1]{fontenc}
\usepackage[latin1]{inputenc}
\usepackage[english]{babel}
\usepackage{amsmath,amssymb,amsthm}
\usepackage{enumitem}
\usepackage{mymacros}
\usepackage[margin=1.5in]{geometry}
\usepackage{hyperref}
\usepackage{color}

\newcommand{\one}{\mathbf{1}}
\newcommand{\mcc}{M\textsuperscript{c}Carthy}

\renewcommand{\MR}[1]{}

\title{Multiplier tests and subhomogeneity of multiplier algebras}

\author[A. Aleman]{Alexandru Aleman}
\address{Lund University, Mathematics, Faculty of Science, P.O. Box 118, S-221 00 Lund, Sweden}
\email{alexandru.aleman@math.lu.se}

\author[M. Hartz]{Michael Hartz}
\address{Fachrichtung Mathematik, Universit\"at des Saarlandes, 66123 Saarbr\"ucken, Germany}
\email{hartz@math.uni-sb.de}
\thanks{M.H. was partially supported by a GIF grant}

\author[J. M\raise.5ex\hbox{c}Carthy]{John E. M\raise.5ex\hbox{c}Carthy}
\address{Department of Mathematics, Washington University in St. Louis, One Brookings Drive,
	St. Louis, MO 63130, USA}
\email{mccarthy@wustl.edu}
\thanks{J.M. was partially supported by National Science Foundation Grant DMS 2054199}

\author[S. Richter]{Stefan Richter}
\address{Department of Mathematics, University of Tennessee, 1403 Circle Drive, Knoxville, TN 37996-1320, USA}
\email{srichter@utk.edu}


\date{\today}

\begin{document}

\begin{abstract}
  Multipliers of reproducing kernel Hilbert spaces can be characterized in
  terms of positivity of $n \times n$ matrices analogous to the classical Pick matrix.
  We study for which reproducing kernel Hilbert spaces it suffices to consider matrices of
  bounded size $n$. We connect this problem to the notion of subhomogeneity of non-selfadjoint operator algebras.
  Our main results show that multiplier algebras of many Hilbert spaces of analytic functions,
  such as the Dirichlet space and the Drury--Arveson space,
  are not subhomogeneous, and hence one has to test Pick matrices of arbitrarily large matrix size $n$.
  To treat the Drury--Arveson space, we show that multiplier algebras of certain weighted Dirichlet spaces on the disc
  embed completely isometrically into the multiplier algebra of the Drury--Arveson space.
\end{abstract}

\maketitle

\section{Introduction}

\subsection{Background}

Multiplication operators on Hilbert spaces of analytic functions appear in several
contexts in analysis. In operator theory, they frequently can be used to model
fairly general classes of operators; see for instance \cite{Agler82,SFB+10}. In complex analysis and harmonic analysis,
multiplication operators provide a functional analytic perspective on concrete function theoretic questions; see
for example \cite{Sarason67,SS61}.
With few exceptions, it is often not an easy task to decide which functions are multipliers
of a given Hilbert space of analytic functions.
Function theoretic descriptions of multipliers exist for spaces such as the Dirichlet space \cite{Stegenga80} and
the Drury--Arveson space \cite{ARS08}, but they tend to be difficult to check in practice.
Nevertheless, there is a very general and well known criterion, which we now describe.

Let $X$ be a non-empty set and let $\cH$ be a reproducing kernel Hilbert space on $X$
with reproducing kernel $K$. We will always assume that reproducing
kernels do not vanish on the diagonal, i.e.\ $K(z,z) \neq 0$ for all $z \in X$.
A function $\varphi: X \to \bC$ is said to be a \emph{multiplier of $\cH$} if $\varphi f \in \cH$
for all $f \in \cH$. In this case, the multiplication operator
\begin{equation*}
  M_\varphi: \cH \to \cH, \quad f \mapsto \varphi \cdot f,
\end{equation*}
is bounded by the closed graph theorem, and the \emph{multiplier norm} $||\varphi||_{\Mult(\cH)}$ of $\varphi$ is
defined to be the operator norm of $M_\varphi$.

The characterization of multipliers alluded to above says that a function $\varphi: X \to \bC$
is a multiplier of $\cH$ of multiplier norm at most $1$ if and only if the function
\begin{equation*}
  X \times X \to \bC, \quad (z,w) \mapsto K(z,w) (1 - \varphi(z) \ol{\varphi(w)}),
\end{equation*}
is positive semi-definite; see \cite[Theorem 5.21]{PR16}. More explicitly,
$||\varphi||_{\Mult(\cH)} \le 1$ if and only if for every $n \in \bN$ and every finite subset $\{z_1,\ldots,z_n\} \subset X$ of $n$ points in $X$, the $n \times n$ matrix
\begin{equation}
  \label{eqn:contractive_multiplier}
  \big[ K(z_i,z_j) (1 - \varphi(z_i) \ol{\varphi(z_j)}) \big]_{i,j = 1}^n
\end{equation}
is positive semi-definite. In some special cases, it is not necessary to check this condition for sets
of arbitrary size $n$.
For instance, if $\cH$ is the Hardy space
on the unit disc, then the multipliers of $\cH$ of norm at most $1$ are precisely the analytic
functions on $\bD$ which are bounded in modulus by $1$.
Hence, if we assume a priori that $\varphi$ is analytic, then
it suffices to test positivity of the matrices
in Equation \eqref{eqn:contractive_multiplier}
for all singleton sets, i.e.\ $n=1$ is suffices. Without the analyticity
assumption, a theorem of Hindmarsh \cite{Hindmarsh68} shows that sets of size $n=3$ suffice.

\subsection{\texorpdfstring{$n$-point multiplier norms}{n-point multiplier norms}}
We ask if similar phenomena occur in other classical reproducing kernel Hilbert spaces.
To study this question, we introduce some terminology. We will also
consider operator valued multipliers. If $\cE$ is an auxiliary Hilbert space, we say
that a function $L: X \times X \to \cB(\cE)$ is \emph{n-point positive}
if for every collection $x_1,\ldots,x_n$ of $n$ (not necessarily distinct) points in $X$, the
$n \times n$ matrix $[L(x_i,x_j)]_{i,j=1}^n$ is a positive operator on $\cB(\cE^n)$.
Thus, $L$ is positive if and only if it is $n$-point positive
for all $n \in \bN$.
We furthermore define
the \emph{$n$-point multiplier norm}
of a function $\Phi: X \to \cB(\cE)$ to be
\begin{equation*}
  ||\Phi||_{\Mult(\cH),n} = \inf \{ C \ge 0 : K(z,w)(C^2 - \Phi(z) \Phi(w)^*) \text{ is $n$-point positive} \},
\end{equation*}
which is understood as $+ \infty$ if no such $C$ exists.
We also adopt the convention that $\|\Phi\|_{\Mult(\cH)} = + \infty$ if $\Phi$ is not a multiplier.

The assumption $K(z,z) \neq 0$ for all $z \in X$ implies that
\begin{equation*}
  \|\Phi\|_{\Mult(\cH),1} = \sup_{z \in X} \|\Phi(z)\|_{\cB(\cE)}.
\end{equation*}
Moreover, it is clear that
\begin{equation*}
  \|\Phi\|_{\Mult(\cH),1} \le
  \|\Phi\|_{\Mult(\cH),2} \le \ldots
  \le \|\Phi\|_{\Mult(\cH)},
\end{equation*}
where the quantities are allowed to be infinity.

We now ask:
\begin{quest}
  \label{quest:many_points}
  Does there exist $n \in \bN$ such that
  \begin{enumerate}
    \item
      $\|\Phi\|_{\Mult(\cH)} = \|\Phi\|_{\Mult(\cH),n}$ for all functions $\Phi: X \to \cB(\cE)$
      and all (finite dimensional) Hilbert spaces $\cE$, or
    \item $\|\varphi\|_{\Mult(\cH)} = \|\varphi\|_{\Mult(\cH),n}$ for all functions $\varphi: X \to \bC$, or
    \item $\|\varphi\|_{\Mult(\cH)} \le C \|\varphi\|_{\Mult(\cH),n}$ for some constant $C > 0$
      and all functions $\varphi: X \to \bC$?
  \end{enumerate}
\end{quest}

Clearly, a positive answer answer to (1) implies a positive answer to (2), which in turn implies
a positive answer to (3).
One might also ask for the seemingly weaker property that there exists $n \in \bN$ so that
$\|\varphi\|_{\Mult(\cH),n} < \infty$ implies $\varphi \in \Mult(\cH)$. However,
one easily checks that
the space of all functions $\varphi: X \to \bC$ with $\| \varphi\|_{\Mult(\cH),n} < \infty$
is a Banach space in the norm $\| \cdot\|_{\Mult(\cH),n}$
(for instance using Equation \eqref{eqn:n-point-alt} below). Hence the closed graph theorem shows
that the seemingly weaker property is in fact equivalent to (3) above.

In Section \ref{sec:123}, we will show that finiteness of $\|\varphi\|_{\Mult(\cH),2}$
implies that $\varphi$ is continuous in an appropriate sense. Moreover,
we generalize Hindmarsh's theorem and prove that if $\cH$ is a space of holomorphic
functions, then finiteness of $\|\varphi\|_{\Mult(\cH),3}$ implies that $\varphi$ is holomorphic.

It is a recurring phenomenon in functional analysis that properties that hold at every matrix level
have strong consequences. This is also the case in Question \ref{quest:many_points}.
We will show that in a fairly general setting of spaces of holomorphic functions on the Euclidean unit
ball $\bB_d \subset \bC^d$, a positive answer to part (1) of Question \ref{quest:many_points}
only happens in the trivial case when the multiplier norm is the supremum norm.
We remark, however, that it is in general not true that $n$-point multiplier norms
are comparable to the supremum norm (see Corollary \ref{cor:dirichlet_two_point} for
the classical Dirichlet space). Instead, we use Arveson's theory of boundary representations
to prove our result addressing part (1) of Question \ref{quest:many_points}.

More precisely, a regular unitarily invariant space is a reproducing kernel Hilbert space on $\bB_d$ with reproducing
kernel of the form
\begin{equation*}
  K(z,w) = \sum_{n=0}^\infty a_n \langle z,w \rangle^n,
\end{equation*}
where $a_0 = 1$, $a_n > 0$ for all $n \in \bN$ and $\lim_{n \to \infty} \frac{a_n}{a_{n+1}} = 1$.
Regular unitarily invariant spaces are a frequently studied class of reproducing kernel
Hilbert spaces; see for instance \cite{GHX04,GRS02}. A discussion of this condition
can be found in \cite{BHM17}. Here, we simply mention
that the Hardy space $H^2(\bD)$, the Drury--Arveson space $H^2_d$, the Dirichlet space $\cD$,
the standard weighted Dirichlet spaces $\cD_a$ as well as standard weighted Dirichlet and Bergman
spaces on $\bB_d$ are regular unitarily invariant spaces. Even the somewhat pathological Salas space \cite{AHM+17a}
belongs to this class.

\begin{thm}
  \label{thm:ci_n_point_intro}
  Let $d \in \bN$ and let $\cH$ be a regular unitarily invariant space on $\bB_d$.
  Then the following are equivalent:
  \begin{enumerate}[label=\normalfont{(\roman*)}]
    \item There exists $n \in \bN$ so that
      $\|\Phi\|_{\Mult(\cH)} = \|\Phi\|_{\Mult(\cH),n}$ for all $\Phi \in \Mult(\cH \otimes \cE)$
      and all finite dimensional Hilbert spaces $\cE$.
    \item $\Mult(\cH) = H^\infty(\bB_d)$ and $\|\Phi\|_{\Mult(\cH \otimes \cE)} = \sup_{z \in \bB_d} \|\Phi(z)\|$
    for all $\Phi \in \Mult(\cH \otimes \cE)$ and all finite dimensional Hilbert spaces $\cE$.
  \end{enumerate}
\end{thm}
Since the $1$-point multiplier norm is the supremum norm, (ii) clearly implies that (i) holds with $n=1$.
A stronger version of this result will be proved in Section \ref{sec:ci_subhom}.
There, we also record applications to concrete spaces of functions on the ball.

Regarding the remaining parts of Question \ref{quest:many_points},
we show that for many concrete spaces on the unit disc or on the unit ball, even Part (3)
of Question \ref{quest:many_points} has a negative answer.
For $a \in (0,\infty)$, let $\cD_a(\bB_d)$ be the reproducing kernel Hilbert space on $\bB_d$
with kernel 
\begin{equation*}
  \frac{1}{(1 - \langle z,w \rangle)^{a}}.
\end{equation*}
We also let $\cD_0(\bB_d)$ be the space corresponding
to the kernel
\begin{equation*}
\frac{1}{\langle z,w \rangle} \log \Big( \frac{1}{1 - \langle z,w \rangle} \Big).
\end{equation*}
Equivalently, for $a > 0$, we have
\begin{equation*}
  \mathcal{D}_a(\mathbb{B}_d) = \Big\{ f = \sum_{\alpha \in \mathbb{N}_0^d} \widehat{f}(\alpha) z^\alpha \in \mathcal{O}(\mathbb{B}_d): \|f\|^2
  = \sum_{\alpha \in \mathbb{N}_0^d} \frac{\alpha! \Gamma(a)}{\Gamma(a+|\alpha|)} |\widehat{f}(\alpha)|^2 < \infty \Big\}
\end{equation*}
and
\begin{equation*}
  \mathcal{D}_0(\mathbb{B}_d) = \Big\{ f = \sum_{\alpha \in \mathbb{N}_0^d} \widehat{f}(\alpha) z^\alpha \in \mathcal{O}(\mathbb{B}_d): \|f\|^2
  = \sum_{\alpha \in \mathbb{N}_0^d} (|\alpha| + 1) \frac{\alpha!}{|\alpha|!}|\widehat{f}(\alpha)|^2 < \infty \Big\};
\end{equation*}
the equivalence can be seen by expanding the reproducing kernel in a power series;
see for instance the proof of Theorem 41 in \cite{ZZ08}.

Then $\cD_1(\bD) = H^2$ is the Hardy space, $\cD_0(\bD)= \cD$ is the classical Dirichlet space and the spaces $\cD_a(\bD) = \cD_a$ are standard weighted
Dirichlet spaces on the unit disc for $a \in (0,1)$. Moreover, $\cD_1(\bB_d) = H^2_d$ is the Drury--Arveson space.

\begin{thm}
  \label{thm:top_n_point_intro}
    Let $d \in \bN$ and let $0 \le a < \frac{d+1}{2}$. Then there do not
    exist a constant $C > 0$ and $n \in \bN$ so that
  \begin{equation*}
    \|\varphi\|_{\Mult(\cD_a(\bB_d))} \le C \| \varphi\|_{\Mult(\cD_a(\bB_d)),n}
  \end{equation*}
  for all $\varphi \in \Mult(\cD_a(\bB_d))$.
\end{thm}

Notice that this result applies in particular to the standard weighted Dirichlet spaces $\cD_a(\bD)$
for $0 \le a < 1$ and to the Drury--Arveson space $H^2_d$ for $d \ge 2$.
We will obtain Theorem \ref{thm:top_n_point_intro} as a special case of Theorem \ref{thm:top_subhom_intro} below.

In fact, for the spaces $\cD_a(\bD)$, we study the $n$-point multiplier norm in more detail.
We show that the $n$-point multiplier norm on $\cD_a(\mathbb{D})$ is comparable to the supremum norm for $a > 0$;
see Corollary \ref{cor:many_points}.
This argument also provides a direct proof of Theorem \ref{thm:top_n_point_intro}
for these spaces.
For the classical Dirichlet space, the $n$-point multiplier norm turns out to be neither comparable
to the supremum norm nor to the full multiplier norm; see Lemma \ref{cor:dirichlet_two_point} and Proposition
\ref{prop:dirichlet_subhom}.

\subsection{Subhomogeneity}
Question \ref{quest:many_points}, at least when asked for a function $\varphi$ or $\Phi$ which is
a priori assumed to be a multiplier,
can be reformulated in representation theoretic terms.
This reformulation connects Question \ref{quest:many_points} to the property
of subhomogeneity of operator algebras, but it is also useful for the sole purpose of understanding
Question \ref{quest:many_points}, as for instance subhomogeneity passes to subalgebras
and is preserved by isomorphisms.

To explain this connection, recall that $\Mult(\cH)$ is a unital (non-selfadjoint) operator algebra,
via identifying a multiplier with its associated multiplication operator.
If $\cA \subset \cB(\cH)$ is an operator algebra, elements of $M_r(\cA)$
can be regarded as operators on $\cH^r$. This identification makes it possible to endow $M_r(\cA)$
with a norm. In the case of multiplier algebras, $M_r(\Mult(\cH))$ is identified with $\Mult(\cH \otimes \bC^r)$.
If $\cA$ and $\cB$ are operator algebras, a linear map $\pi: \cA \to \cB$ induces
linear maps $\pi^{(r)} : M_r(\cA) \to M_r(\cB)$, defined by applying $\pi$ entrywise.
The linear map $\pi$ is said to be completely contractive if each $\pi^{(r)}$ is contractive,
and completely isometric if each $\pi^{(r)}$ is isometric.

Suppose now that $F = \{z_1,\ldots,z_n\}$ is a finite subset of $X$ and let $\cH \big|_F$
be the restriction of $\cH$ to $F$, i.e.\ the reproducing kernel Hilbert space
on $F$ with reproducing kernel $K \big|_{F \times F}$ (see \cite[Part I, Section 6]{Aronszajn50}).
Then
\begin{equation}
  \label{eqn:n-point-alt}
  \|\Phi\|_{\Mult(\mathcal{H}),n} = \sup_{|F| \le n} \| \Phi \big|_F \|_{\Mult( (\mathcal{H} |_F) \otimes \mathcal{E})}
\end{equation}
for every function $\Phi: X \to \mathcal{B}(\mathcal{E})$.
Thus, if we consider the unital completely contractive (u.c.c.) homomorphism
\begin{equation*}
  \pi_F: \Mult(\cH) \to \Mult(\cH \big|_F), \quad \varphi \mapsto \varphi \big|_F,
\end{equation*}
then for $\Phi \in \Mult(\mathcal{H} \otimes \mathbb{C}^r)$, we have
\begin{equation}
  \label{eqn:subhom}
  \|\Phi\|_{\Mult(\cH),n} = \sup_{|F| \le n} \| \pi_F^{(r)} (\Phi) \|.
\end{equation}
On the other hand,
\begin{equation*}
  \|\Phi\|_{\Mult(\cH)} = \sup_{|F| < \infty} \| \pi_F^{(r)} (\Phi) \|.
\end{equation*}

Note that $\dim(\cH \big|_F) \le |F|$. In particular, we see that $\Mult(\cH)$
is a \emph{residually finite-dimensional} operator algebra, which means that
for every $r \in \bN$ and every $\Phi \in M_r(\Mult(\cH))$,
\begin{equation*}
  \|\Phi\| = \sup \{ \|\pi^{(r)} (\Phi) \|: \pi: \Mult(\cH) \to M_n \text{ is a u.c.c.\ homomorphism}, n \in \bN \}.
\end{equation*}
Question \ref{quest:many_points} is then closely related to the question whether in the above supremum,
it suffices to consider representations into $M_n$ for uniformly bounded values of $n$.

The notion of residual finite-dimensionality
was originally defined for $C^*$-algebras and plays a crucial role in this theory;
see for instance \cite{Archbold95} and the references given in the introduction.
Residually finite-dimensional operator algebras of functions were studied in \cite{MP10},
where it was also observed that multiplier algebras are residually finite-dimensional. 
For residual finite-dimensionality of more general non-selfadjoint operator algebras,
see \cite{CR18}.

In the theory of $C^*$-algebras, a significant strengthening of the property of residual finite-dimensionality
is the notion of \emph{subhomogeneity}; see for instance \cite[Section IV.1.4]{Blackadar06}.
A unital $C^*$-algebra $\cA$ is said to be $n$-subhomogeneous if every irreducible $*$-representation of $\cA$
acts on a Hilbert space of dimension at most $n$. For $C^*$-algebras,
this is equivalent to saying that for all $a \in \cA$,
\begin{equation*}
  \|a\| = \sup \{ \|\pi(a) \| : \pi: \cA \to M_k \text{ is a unital $*$-homomorphism, } k \le n \};
\end{equation*}
see the discussion at the beginning of Section \ref{sec:ci_subhom}.

We therefore make the following definitions.
\begin{defn}
  Let $\cA$ be a unital operator algebra and let $n \in \bN$.
  \begin{enumerate}
    \item $\cA$ is \emph{completely isometrically} $n$-subhomogeneous if for all $r \in \bN$ and all $A \in M_r(\cA)$,
      \begin{equation*}
        \|A\| = \sup\{ \|\pi^{(r)}(A) \| \},
      \end{equation*}
      where the supremum is taken over all u.c.c.\ homomorphisms $\pi: \cA \to M_k$ and all $k \le n$.
    \item $\cA$ is \emph{isometrically} $n$-subhomogeneous if for all $a \in \cA$,
      \begin{equation*}
        \|a\| = \sup\{ \|\pi(a) \| \},
      \end{equation*}
      where the supremum is taken over all unital contractive homomorphisms $\pi: \cA \to M_k$ and all $k \le n$.
    \item $\cA$ is \emph{topologically} $n$-subhomogeneous if there exist constants $C_1,C_2 > 0$ so that
      for all $a \in \cA$,
      \begin{equation*}
        \|a\| \le C_1 \sup\{ \|\pi(a) \| \},
      \end{equation*}
      where the supremum is taken over all unital homomorphisms $\pi: \cA \to M_k$ with $\|\pi\| \le C_2$
      and all $k \le n$.
  \end{enumerate}
  We say that $\cA$ is completely isometrically / isometrically / topologically subhomogeneous
  if it is completely isometrically / isometrically / topologically $n$-subhomogeneous
  for some $n \in \bN$.
\end{defn}

Thus, Question \ref{quest:many_points} leads us to ask the following
more general question for a reproducing kernel Hilbert space $\cH$.

\begin{quest}
  \label{quest:subhomogeneous}
  Does there exist $n \in \bN$ so that
  \begin{enumerate}
    \item $\Mult(\cH)$ is completely isometrically $n$-subhomogeneous, or
    \item $\Mult(\cH)$ is isometrically $n$-subhomogeneous, or
    \item $\Mult(\cH)$ is topologically $n$-subhomogeneous?
  \end{enumerate}
\end{quest}

Equation \eqref{eqn:subhom} shows that for $j=1,2,3$, a positive answer
to part (j) of Question \ref{quest:many_points} implies a positive answer
to part (j) of Question \ref{quest:subhomogeneous} (for the same value of $n$).

We then show the following stronger version of Theorem \ref{thm:ci_n_point_intro}
in Theorem \ref{thm:ci_subhom}.

\begin{thm}
  \label{thm:ci_subhom_intro}
  Let $\cH$ be a regular unitarily invariant space on $\bB_d$.
  Then $\Mult(\cH)$ is completely isometrically subhomogeneous if and only if $\Mult(\cH) = H^\infty(\bB_d)$
  completely isometrically.
\end{thm}
The discussion above shows that this result will establish the non-trivial implication (i) $\Rightarrow$ (ii) in Theorem \ref{thm:ci_n_point_intro},
as (i) in Theorem \ref{thm:ci_n_point_intro} in particular implies that $\Mult(\cH)$ is completely isometrically subhomogeneous.

Theorem \ref{thm:top_n_point_intro} also holds in the following stronger sense.

\begin{thm}
  \label{thm:top_subhom_intro}
  Let $d \in \bN$ and let $0 \le a < \frac{d+1}{2}$. Then $\Mult(\cD_a(\bB_d))$ is not topologically
  subhomogeneous.
\end{thm}

Just as Theorem \ref{thm:top_n_point_intro}, this result
applies in particular to the standard weighted Dirichlet spaces $\cD_a(\bD)$
for $0 \le a < 1$ and to the Drury--Arveson space $H^2_d$ for $d \ge 2$.
Theorem \ref{thm:top_subhom_intro} will be proved in Corollary \ref{cor:between_spaces_subhom}.
The cases of $\cD_a(\bD)$ for $0 < a < 1$, $\cD_0(\bD)$ and $H^2_d$ are already contained in
Corollary \ref{cor:not_subhom}, Proposition \ref{prop:dirichlet_subhom}
and Corollary \ref{cor:da_not_subhom},
respectively.

To show Theorem \ref{thm:top_subhom_intro} for the Drury--Arveson space, we establish an embedding result
for multiplier algebras of certain weighted Dirichlet spaces on the unit disc, which
may be of independent interest; see Proposition \ref{prop:mult_embed} for the precise statement. We exhibit two more consequences of the embedding result.
In particular, we show that a certain sufficient condition for membership in the multiplier
algebra of the Drury--Arveson space of \cite{AHM+17c} is not necessary.
This was also proved, using different techniques, in a recent paper of Fang and Xia \cite{FX20}.

The remainder of this paper is organized as follows. In Section \ref{sec:123}, we study the $n$-point multiplier norm
of functions for small values of $n$. In particular, we establish a generalization of Hindmarsh's theorem
for spaces of holomorphic functions.  Section \ref{sec:n-point} contains the result that the $n$-point multiplier
norm on $\cD_a$ is comparable to the supremum norm for $a \in (0,1)$, but not for the classical Dirichlet space.
In Section \ref{sec:ci_subhom}, we establish our results on completely isometric subhomogeneity of multiplier algebras.
Section \ref{sec:top_sub_disc} addresses topological subhomogeneity of multiplier algebras on the unit disc.
In Section \ref{sec:subhom_ball}, we deduce from our results on the unit disc results
about topological subhomogeneity
of multiplier algebras on the unit ball. Section \ref{sec:embedding} contains the embedding of multiplier
algebras of weighted Dirichlet spaces into the multiplier algebra of the Drury--Arveson space, from which
we deduce that the multiplier algebra of the Drury--Arveson space is not topologically subhomogeneous.
In Section \ref{sec:embedding_applications}, we give two more applications of our embedding result established in Section \ref{sec:embedding}. Section \ref{sec:spaces_between} deals with topological subhomogeneity
of multiplier algebras of spaces between the Drury--Arveson space and the Hardy space.
Finally, we close the article in Section \ref{sec:questions} with some questions.

We make one final remark regarding notation. If $f,g$ are two functions taking non-negative
real values, we write $f \lesssim g$ if there exists a constant $C > 0$ so that $f \le C g$.
If $f \lesssim g$ and $g \lesssim f$, we write $f \approx g$. Finally, the notation
$f \sim g$ as $x \to x_0$ means that $\lim_{x \to x_0} \frac{f(x)}{g(x)} = 1$.

\subsection*{Acknowledgements} The second named author is grateful for valuable
discussions with Ken Davidson regarding Proposition \ref{prop:mult_embed}. He also
thanks J\"org Eschmeier for asking two questions that led to the results in Section \ref{sec:embedding_applications}.

\section{Boundedness, continuity and analyticity}
\label{sec:123}

In this section, we study the $n$-point multiplier norm of a function
for small values of $n$. Throughout, $\cH$
denotes a reproducing kernel Hilbert space on a set $X$ with kernel $K$
satisfying $K(z,z) \neq 0$ for all $z \in X$.
As mentioned in the introduction,
  \begin{equation*}
    ||\varphi||_{\Mult(\cH),1} = \sup_{z \in X} |\varphi(z)|
  \end{equation*}
  for all functions $\varphi: X \to \bC$.

The $2$-point norm can be understood in terms of a pseudo-metric induced
by $\cH$, which is for example studied in \cite{ARS+11}.
This pseudo-metric is defined by
\begin{equation*}
  \delta_{\cH}(z,w) = \Big( 1- \frac{|K(z,w)|^2}{K(z,z)K(w,w)} \Big)^{1/2}
\end{equation*}
for $z,w \in X$; see \cite[Lemma 9.9]{AM02} for a proof of the triangle inequality.
Recall that the Hardy space $H^2$ is the reproducing kernel Hilbert space on $\mathbb{D}$ whose
reproducing kernel is the Szeg\H{o} kernel $S(z,w) = \frac{1}{1 - z \overline{w}}$.
In this case, $\delta_{H^2}$ is the classical pseudohyperbolic metric $\delta$
on $\bD$, that is,
\begin{equation*}
  \delta_{H^2}(z,w) = \delta(z,w) = \Big| \frac{z - w}{1 - \ol{w} z} \Big|;
\end{equation*}
see \cite[Equation 9.8]{AM02}.

We now show that the $2$-point multiplier norm is closely related to the metric $\delta_{\cH}$.
We say that $K$ is non-vanishing if $K(z,w) \neq 0$ for all choices of $z,w$.

\begin{prop}
  Let $\cH$ be a reproducing kernel Hilbert space on $X$ with non-vanishing kernel $K$.
  Let $\varphi: X \to \bC$ be a function. Then
  \begin{equation*}
    ||\varphi||_{\Mult(\cH),2} \le 1
  \end{equation*}
  if and only if $\varphi$ is constant of modulus at most $1$ or $\varphi$ maps $X$ into $\bD$
  and satisfies
  \begin{equation*}
    \delta(\varphi(z),\varphi(w)) \le \delta_{\cH}(z,w)
  \end{equation*}
  for all $z, w \in X$.
\end{prop}

\begin{proof}
  Let $z,w \in X$. By Sylvester's criterion, the matrix
  \begin{equation*}
    \begin{bmatrix}
      K(z,z) ( 1 - |\varphi(z)|^2) & K(z,w) (1 - \varphi(z) \ol{\varphi(w)}) \\
      K(w,z) ( 1 - \varphi(w) \ol{\varphi(z)}) & K(w,w) (1 - |\varphi(w)|^2)
    \end{bmatrix}
  \end{equation*}
  is positive if and only if $|\varphi(z)| \le 1$ and $|\varphi(w)| \le 1$
  and its determinant is non-negative. If $\varphi(z),\varphi(w) \in \bD$,
  then the determinant of this matrix
  is non-negative if and only if
  \begin{equation*}
    \frac{|K(z,w)|^2}{K(z,z)K(w,w)} \le \frac{(1- |\varphi(z)|^2)(1 - |\varphi(w)|^2)}{
      |1 - \varphi(z) \ol{\varphi(w)}|^2},
  \end{equation*}
  and since $\delta_{H^2} = \delta$, this last inequality is equivalent to
  \begin{equation*}
    \delta(\varphi(z),\varphi(w)) \le \delta_{\cH}(z,w).
  \end{equation*}

  This shows that if $\varphi$ maps $X$ into $\bD$ and is a $\delta_{\cH}-\delta$-contraction,
  then $||\varphi||_{\Mult(\cH),2} \le 1$. Clearly, constant functions $\varphi$
  of modulus at most $1$ also have $2$-point norm at most $1$.

  Conversely, if $\varphi$ has $2$-point norm at most $1$ and is not constant,
  then since $K$ does not vanish, \cite[Lemma 2.2]{AHM+17a}
  implies that $\varphi$ maps $X$ into $\bD$. By the first paragraph,
  $\varphi$ is a $\delta_{\cH}-\delta$-contraction.
\end{proof}

Next, we study the $3$-point multiplier norm.
If $\cH = H^2$, the Hardy space on the disc,
then every function $\varphi: \bD \to \bC$ with $||\varphi||_{\Mult(\cH),3} < \infty$
is analytic by a theorem of Hindmarsh \cite{Hindmarsh68}.
This result can be generalized. We begin with a lemma.

\begin{lem}
  \label{lem:analytic_szego}
  Let $\cH$ be a reproducing kernel Hilbert space of analytic functions on $\bD$ with kernel $K$.
  Let $S$ denote the Szeg\H{o} kernel.
  Then there exists $r > 0$ such that $K(rz, rw) \neq 0$ for all $z,w \in \mathbb{D}$ and
  such that
  \begin{equation*}
    \bD \times \bD \to \bC, \quad
    (z,w) \mapsto \frac{S(z,w)}{K(r z,r w)},
  \end{equation*}
  is positive semi-definite.
\end{lem}

\begin{proof}
  In the first step, we show that if $(b_n)$ is a sequence
  of real numbers with $b_0 > 0$ such that the power series
  $\sum_{n=0}^\infty b_n t^n$ has a positive radius of convergence and if we define
  \begin{equation*}
    L(z,w) = \sum_{n=0}^\infty b_n (z \ol{w})^n,
  \end{equation*}
  then there exists $r_0 > 0$ such that
  \begin{equation*}
    (z,w) \mapsto S(z,w) L(r z, rw)
  \end{equation*}
  is positive semi-definite on $\bD \times \bD$ for all $r \in (0,r_0)$. (This will already
  prove the lemma in the case when $K$ has circular symmetry, by taking $L = 1/K$.)

  To establish this claim, observe that for sufficiently small $r > 0$, the identity
  \begin{equation*}
    S(z,w) L(r z , r w) = \sum_{n=0}^\infty (z \ol{w})^n \sum_{n=0}^\infty b_n r^n (z \ol{w})^n
    = \sum_{n=0}^\infty (z \ol{w})^n \sum_{k=0}^n b_k r^k
  \end{equation*}
  holds,
  so it suffices to find $r_0 > 0$ such that $\sum_{k=0}^n b_k r^k \ge 0$ for all $n \in \bN$
  and all $r \in (0,r_0)$.
  Since $\sum_{n=0}^\infty b_n t^n$ has a positive radius of convergence and since $b_0 > 0$, there exists
  $r_0 > 0$ such that $\sum_{k=1}^\infty |b_k| r_0^k \le b_0$. Thus,
  \begin{equation*}
    \sum_{k=0}^n b_k r^k \ge b_0 - \sum_{k=1}^n |b_k| r^k \ge b_0 - \sum_{k=1}^\infty |b_k| r_0^k \ge 0
  \end{equation*}
  for all $r \in (0,r_0)$, as desired.

  Having established the claim, suppose
  now that $K$ is the reproducing kernel of a Hilbert space of analytic functions
  on $\bD$.
  We will show that there exist $s > 0$ and a sequence of real numbers $(b_n)$ with $b_0 > 0$
  such that $\sum_{n=0}^\infty b_n t^n$ has a positive radius of convergence, $K(sz ,s w) \neq 0$ for all $z,w \in \mathbb{D}$, and
  \begin{equation*}
    (z,w) \mapsto K^{-1}( s z, s w) - \sum_{n=0}^\infty b_n (z \ol{w})^n
  \end{equation*}
  is positive semi-definite in a neighborhood of $(0,0)$. Once this is accomplished,
  the first part shows that there exists $r_0 > 0$ such that $L(r z,rw) S(z,w) \ge 0$
  for all $r \in (0,r_0)$,
  where $L(z,w) = \sum_{n=0}^\infty b_n (z \ol{w})^n$. Hence, in the order
  induced by positivity, we obtain by the Schur product theorem for sufficiently small $r > 0$
  the inequality
  \begin{equation*}
    K^{-1}(s r z, s r w) S(z,w) \ge L(r z, rw) S(z,w) \ge 0,
  \end{equation*}
  which finishes the proof.

  To construct the sequence $(b_n)$, observe that since $\cH$ consists
  of holomorphic functions, $K$ is holomorphic in the first variable
  and conjugate holomorphic in the second variable. Moreover, $K(0,0) > 0$ by our standing assumption that reproducing kernels
  do not vanish on the diagonal.
  By Hartogs' theorem, $K$ is jointly continuous, so $K$ is non-vanishing in neighborhood of $0$, and there exist complex numbers $(c_{n m})$ such that
  \begin{equation*}
    K^{-1} (z,w) = \sum_{n,m=0}^\infty c_{n m} z^n \ol{w}^m
  \end{equation*}
  for $z,w$ in a neighborhood of $0$.
  Since $K(z,w) = \ol{K(w,z)}$, it follows that $c_{n m} = \ol{c_{m n}}$ for all $m,n \in \bN_0$.
  In particular, each $c_{n n}$ is real. Moreover, $c_{0 0} > 0$ as $K(0,0) > 0$.
  By replacing $K$ with a positive multiple of $K$, we may
  assume that $c_{0 0} = 1$.
  Moreover, by replacing $K(z,w)$ with $K(sz , sw)$ for suitable $s > 0$, we may assume that $|c_{n m}| \le 1$
  whenever $(n,m) \neq (0,0)$.
  Then
  \begin{equation}
    \label{eqn:K_inverse}
    K^{-1}(z,w)
    = 1 + \sum_{n = 1}^\infty c_{n n} (z \ol{w})^n
    + \sum_{n < m} (c_{n m} z^n \ol{w}^m + \ol{c_{n m}} z^m \ol{w}^n).
  \end{equation}
  To treat the last sum, observe that if $f$ is a function on $\bD$, then for all $\varepsilon > 0$,
  \begin{equation*}
    f(z) + \ol{f(w)} = \left( \varepsilon + \frac{f(z)}{\varepsilon} \right)
    \left( \varepsilon + \frac{\ol{f(w)}}{\varepsilon} \right) - \varepsilon^2 - \frac{f(z) \ol{f(w)}}{\varepsilon^2},
  \end{equation*}
  hence
  \begin{equation*}
    f(z) + \ol{f(w)} \ge - \varepsilon^2 - \frac{f(z) \ol{f(w)}}{\varepsilon^2}.
  \end{equation*}
  Using this observation with $f(z) = \ol{c_{n m}} z^{m -n}$ and $\varepsilon = \varepsilon_m$, to be determined later, as well as the Schur product theorem, we see that if $n < m$, then
  \begin{align*}
    c_{n m} z^n \ol{w}^m + \ol{c_{n m}} z^m \ol{w}^n
    &= z^n \ol{w}^n ( c_{n m} \ol{w}^{m -n} + \ol{c_{n m}} z^{m - n}) \\
    &\ge z^n \ol{w}^n (- \varepsilon_m^2 - \varepsilon_{m}^{-2} |c_{n m}|^2 z^{m -n} \ol{w}^{m - n}) \\
    &\ge - \varepsilon_m^2 z^n \ol{w}^n - \varepsilon_m^{-2} z^{m} \ol{w}^m.
  \end{align*}
  Consequently, if $\varepsilon_m^2 = 2^{-m-1}$, then
  \begin{align*}
    \sum_{n < m} (c_{n m} z^n \ol{w}^m + \ol{c_{n m}} z^m \ol{w}^n)
    &\ge - \sum_{n < m} (\varepsilon_m^2 z^n \ol{w}^n + \varepsilon_m^{-2} z^m \ol{w}^m) \\
    &= - \sum_{n =0}^{\infty} \sum_{m=n+1}^\infty  \varepsilon_m^2 z^n \ol{w}^n
    - \sum_{m=0}^\infty \sum_{n=0}^{m-1} \varepsilon_{m}^{-2} z^m \ol{w}^m \\
    &\ge - \frac{1}{2} \sum_{n=0}^\infty (z \ol{w})^n - \sum_{m=0}^\infty m 2^{m+1} z^m \ol{w}^m,
  \end{align*}
  where all sums converge absolutely for $(z,w)$ in a neighborhood of $0$.
  Combining this estimate with Equation \eqref{eqn:K_inverse}, we see that
  \begin{equation*}
    K^{-1}(z,w) \ge \frac{1}{2} - \sum_{n=1}^\infty \Big(\frac{3}{2} + n 2^{n+1}\Big) (z \ol{w})^n,
  \end{equation*}
  so if we set $b_0 = \frac{1}{2}$ and $b_n = -\tfrac{3}{2} - n 2^{n+1}$, then
  $(b_n)$ satisfies all desired properties.
\end{proof}

We can now generalize Hindmarsh's theorem \cite{Hindmarsh68}.
The original theorem is obtained by taking $\cH = H^2$.

\begin{prop}
  \label{prop:Hindmarsh_general}
  Let $\cH$ be a reproducing kernel Hilbert space of analytic functions
  on an open domain $\Omega \subset \bC^d$. Then every function
  $\varphi: \Omega \to \bC$ with $||\varphi||_{\Mult(\cH),3} < \infty$
  is analytic on $\Omega$.
\end{prop}

\begin{proof}
  Suppose that $||\varphi||_{\Mult(\cH),3} \le 1$. We will show that $\varphi$
  is analytic in each variable separately.
  To this end, let $w \in \Omega$ and $j \in \{1,\ldots,d\}$,
  and choose $s > 0$ such that $w + s \bD e_j \subset \Omega$,
  where $e_1,\ldots,e_d$ is the standard basis of $\mathbb{C}^d$.
  Let $D = w + s \bD e_j$, define
  \begin{equation*}
    \iota: \bD \to D,  \quad t \mapsto w + s t e_j,
  \end{equation*}
  and let $k$ be the kernel on $\bD$ given by $k(z,w) = K(\iota(z),\iota(w))$.
  Then the reproducing kernel Hilbert space $\cH(k)$ on $\bD$ with kernel $k$
  consists of analytic functions. Let $\psi = \varphi \circ \iota$. Then
  \begin{equation*}
    ||\psi||_{\Mult(\cH(k)),3} \le ||\varphi||_{\Mult(\cH),3} \le 1.
  \end{equation*}
  Lemma \ref{lem:analytic_szego} shows that there exists $r > 0$ such that
  \begin{equation*}
    S(z,w) k^{-1}(r z,rw) \ge 0,
  \end{equation*}
  where $S$ denotes the Szeg\H{o} kernel.
  Since $||\psi||_{\Mult(\cH(k)),3} \le 1$, the function
  \begin{equation*}
    (z,w) \mapsto k(rz,rw) (1 - \psi(rz) \ol{\psi(rw)})
  \end{equation*}
  is $3$-point positive, hence an application of the Schur product theorem yields that
  \begin{equation*}
    (z,w) \mapsto S(z,w) (1 - \psi(r z) \ol{\psi(r w)})
  \end{equation*}
  is $3$-point positive as well. In this setting, Hindmarsh's theorem \cite{Hindmarsh68}, see also
  \cite[Theorem III.2]{Donoghue74}, implies that $z \mapsto \psi(r z)$ is analytic in $\bD$.
  (Hindmarsh's theorem concerns functions in the upper half plane, but a routine application
  of the Cayley transform yields the corresponding statement in the unit disc.)
  Therefore, $\psi$ is analytic in a neighborhood of the origin, so that $\varphi$ is analytic
  in the $j$-th variable in a neighborhood of $w$.
\end{proof}

\section{\texorpdfstring{The $n$-point norm for spaces on $\bD$}{The n-point norm for spaces on D}}
\label{sec:n-point}

In this section, we study the $n$-point multiplier norm for spaces of holomorphic
functions on the unit disc.
Our first goal is to show that Question \ref{quest:many_points} has a negative answer for the weighted
Dirichlet spaces $\cD_a$, i.e.\ no $n$-point multiplier norm is comparable
to the full multiplier norm for these spaces. While we will establish a more general result concerning subhomogeneity of $\Mult(\cD_a)$ in Section \ref{sec:top_sub_disc},
we consider this easier question first as it illustrates the ideas that will be used later.

The following lemma shows that for
certain spaces on $\bD$,
the $n$-point multiplier norm of an analytic
function is comparable to the supremum norm.
In the sequel, we let $\Aut(\bD)$ denote the group
of conformal automorphisms of $\bD$.

\begin{lem}
  \label{lem:n_point}
  Let $\cH$ be a reproducing kernel Hilbert
  space on $\bD$.
  Suppose that
  every conformal automorphism of $\bD$ is a multiplier
  on $\cH$ and that
  there exists a constant $C > 0$ such that
  \begin{equation*}
    ||\theta||_{\Mult(\cH)} \le C
  \end{equation*}
  for all $\theta \in \Aut(\bD)$.
  Then
  \begin{equation*}
    ||f||_{\infty} \le ||f||_{\Mult(\cH),n} \le C^{n-1} ||f||_{\infty}
  \end{equation*}
  for all $f \in H^\infty$ and all $n \ge 1$.
\end{lem}

\begin{proof}
  The first inequality always holds, so it suffices to show the second one.
  To this end, suppose that $f \in H^\infty$ with $||f||_\infty \le 1$.
  Let $z_1,\ldots,z_n \in \bD$. We wish to show that
  \begin{equation*}
    \big[K(z_i,z_j) (C^{2 n-2} - f(z_i) \ol{f(z_j)}) \big]_{i,j=1}^n
  \end{equation*}
  is positive.
  Since
  $f$ belongs to the unit ball of $H^\infty$, there exists by classical Nevanlinna-Pick
  interpolation  a finite Blaschke product $B$ of degree at most $n-1$ and a complex
  number $\lambda$ with $|\lambda| \le 1$ such that
  $\lambda B(z_i) = f(z_i)$ for $1 \le i \le n$; see \cite[Section I.2]{Garnett07} or \cite[Theorem 6.15]{AM02}.
  Since $B$ is a product
  of at most $n-1$ conformal automorphisms of $\bD$, the assumption on $\cH$
  implies that $ \lambda B$ is a multiplier of $\cH$ of norm at most $C^{n-1}$. Therefore,
  \begin{equation*}
    \big[ K(z_i,z_j) ( C^{2n-2} - f(z_i) \ol{f(z_j)}) \big] =
    \big[ K(z_i,z_j) ( C^{2n-2} - \lambda B(z_i) \ol{\lambda B(z_j)}) \big]
  \end{equation*}
  is positive, as desired.
\end{proof}

Using basic results from operator space theory, it is possible
to extend the preceding lemma to operator-valued multipliers.
If $\cE$ is an auxiliary Hilbert space,
let $H^\infty(\cB(\cE))$ denote the space of bounded $\cB(\cE)$-valued
holomorphic functions on $\bD$, equipped with the supremum norm
\begin{equation*}
  \|\Phi\|_{\infty} = \sup_{z \in \bD} \|\Phi(z)\|_{\cB(\cE)}.
\end{equation*}
Thus, $H^\infty(\cB(\cE)) = \Mult(H^2 \otimes \cE)$, with equality of norms.

\begin{lem}
  \label{lem:n_point_vector}
  Assume the setting of Lemma \ref{lem:n_point} and let $\cE$ be an auxiliary Hilbert space.
  Then
  \begin{equation*}
    \|\Phi\|_\infty \le \|\Phi\|_{\Mult(\cH),n} \le n C^{n-1} \|\Phi\|_{\infty}
  \end{equation*}
  for all $\Phi \in H^\infty(\cB(\cE))$ and all $n \ge 1$.
\end{lem}

\begin{proof}
  Once again, the first inequality always holds.
  To prove the second inequality,
  a straightforward approximation argument shows that it suffices to consider
  finite-dimensional spaces $\cE$.
  
  Let $n \in \bN$, let $F \subset \bD$ with $|F| \le n$ and consider the restriction mapping
  \begin{equation*}
    R: H^\infty \to \Mult(\cH|_F), \quad \varphi \mapsto \varphi \big|_F.
  \end{equation*}
  Lemma \ref{lem:n_point} implies that $R$ is bounded with norm at most $C^{n-1}$.
  Since 
  \begin{equation*}
    \dim( \Mult(\cH |_F)) \le n,
  \end{equation*}
  a basic result from operator space theory implies that $R$
  is completely bounded with completely bounded norm at most $n C^{n-1}$;
  see for instance \cite[Corollary 2.2.4]{ER00}.
  In other words,
  if $k \ge 1$ and $\Phi \in H^\infty(M_k)$, then $\|\Phi\|_{\Mult( (\cH|_F) \otimes \bC^k)} \le n C^{n-1} \|\Phi\|_\infty$.
  Taking the supremum over all finite subsets $F$ of $\bD$ of cardinality at most $n$
  therefore yields the second inequality.
\end{proof}

Lemma \ref{lem:n_point} implies a negative answer to Question \ref{quest:many_points} for the weighted
Dirichlet spaces $\cD_a$, where $a \in (0,1)$. In fact, we obtain the following more precise statement.

\begin{cor}
  \label{cor:many_points}
  Let $a \in (0,1)$. Then there exists a constant $C > 0$ so that
  for all $n \in \bN$ and all $f \in H^\infty$,
  \begin{equation*}
    \|f\|_\infty \le \|f\|_{\Mult(\cD_a),n} \le C^{n-1} \|f\|_\infty.
  \end{equation*}
  More generally, for all $n \in \bN$, all auxiliary Hilbert spaces $\cE$
  and all $\Phi \in H^\infty(\cB(\cE))$,
  \begin{equation*}
    \|\Phi\|_\infty \le \|\Phi\|_{\Mult(\mathcal{D}_a),n} \le n C^{n-1} \|\Phi\|_{\infty}.
  \end{equation*}
  In particular,
there do not exist a constant $C> 0$ and $n \in \bN$ so that
\begin{equation*}
  \|\varphi\|_{\Mult(\cD_a)} \le C \|\varphi\|_{\Mult(\cD_a),n}
\end{equation*}
for all $\varphi \in \Mult(\cD_a)$.
\end{cor}

\begin{proof}
  We show that $\cD_{a}$ satisfies the assumptions of Lemma \ref{lem:n_point},
  which will prove the first two statements (by Lemma \ref{lem:n_point} and Lemma \ref{lem:n_point_vector}). The final statement then
  follows from the well known fact that the multiplier norm of $\cD_{a}$ is not equivalent to the supremum norm.
  Indeed, using the explicit formula for the norm in $\mathcal{D}_a$ from the introduction, we see that
  \begin{equation*}
    \|z^n\|_{\Mult(\cD_a)}^2 \ge \|z_n\|_{\cD_a}^2 = \frac{n! \Gamma(a)}{\Gamma(n+a)} \approx (n+1)^{1-a},
  \end{equation*}
  which tends to infinity as $n \to \infty$; see for instance \cite{Wendel48} for the asymptotic relation.

  Using the special form of the reproducing kernel of $\mathcal{D}_a$
  and a familiar identity for disc automorphisms \cite[Theorem 2.2.5 (2)]{Rudin08}, it is not hard to see that
  $\Mult(\cD_a)$ is isometrically invariant under compositions with conformal
  automorphisms, that is, if $\varphi \in \Mult(\cD_{a})$ and if $\theta \in \Aut(\bD)$,
  then $\varphi \circ \theta \in \Mult(\cD_{a})$ and $\| \varphi \circ \theta\|_{\Mult(\cD_{a})}
  = \|\varphi\|_{\Mult(\cD_{a})}$; see for instance the easy implication of \cite[Corollary 4.4]{Hartz17a}.
  In particular, $\|\theta\|_{\Mult(\cD_a)} = \|z\|_{\Mult(\cD_{a})}$ for all $\theta \in \Aut(\bD)$,
  so that $\cD_{a}$ satisfies the assumptions of Lemma \ref{lem:n_point}.
\end{proof}

The classical Dirichlet space $\cD$ requires a more careful analysis, as $\Aut(\bD)$
is not a bounded subset of $\Mult(\cD)$. In fact, we will shortly see that the $2$-point
multiplier norm on the Dirichlet space is not equivalent to the supremum norm on $\bD$,
and hence provides more information than the $1$-point norm.
This result will be a consequence of the following estimate of the derivative in terms of the $2$-point norm
on the Dirichlet space, which is better than the classical estimate
in the Schwarz--Pick lemma by a factor of $\log( \frac{1}{1 - |z|})^{1/2}$.

\begin{lem}
  \label{lem:der_dirichlet}
  There exists a constant $C > 0$ so that
  for all $f \in \cO(\bD)$ and all $z \in \bD$,
  \begin{equation}
    \label{eqn:der_dirichlet}
    |f'(z)| \le C  \frac{\|f\|_{\Mult(\cD),2}}{(1 - |z|) \log(\frac{1}{1- |z|})^{1/2}}.
  \end{equation}
\end{lem}

\begin{proof}
  By Proposition 18 (a) of \cite{BS84}, there exists a constant $C > 0$ so that
  for all $f \in \Mult(\cD)$, inequality \eqref{eqn:der_dirichlet} holds with $\|f\|_{\Mult(\cD)}$
  in place of $\|f\|_{\Mult(\cD),2}$. Indeed, this follows from the fact that the map
  $f \mapsto f'$ is a bounded linear map from $\Mult(\cD)$ into $\Mult(\cD,L^2_a)$, combined with
  the standard estimate
  for multipliers between reproducing kernel Hilbert spaces.

  Let now $f \in \cO(\bD)$ with $\|f\|_{\Mult(\cD),2} \le 1$ and let $z \in \bD$.
  By the complete Nevanlinna--Pick property of $\cD$ (see, for instance, \cite[Corollary 7.41]{AM02}), there exists for every $w \in \bD$
  a multiplier $\varphi_w \in \Mult(\cD)$ with $\|\varphi_w\|_{\Mult(\cD)} \le 1$
  and $\varphi_w(z) = f(z), \varphi_w(w) = f(w)$. Thus,
  \begin{equation*}
    |f(z) - f(w)| = |\varphi_w(z) - \varphi_w(w)|
    \le \sup_{t \in [0,1]} |\varphi_w'(z + t (w - z))| \, |z - w|.
  \end{equation*}
  Applying the estimate for $\varphi_w'$ explained in the preceding paragraph, we find that
  \begin{equation*}
    \frac{ |f(z) - f(w)|}{|z - w|} \le \sup_{t \in [0,1]} \frac{C}{ ( 1 - |z + t (w - z)|) \log(\frac { 1}{1 - | z + t (w - z)|})^{1/2}}.
  \end{equation*}
  Taking the limit $w \to z$, the lemma follows.
\end{proof}

The result about the $2$-point norm that was alluded to above is an immediate consequence.

\begin{cor}
  \label{cor:dirichlet_two_point}
  The norm $\|\cdot\|_{\Mult(\cD),2}$ and the supremum norm on $\bD$ are not equivalent
  on the space of functions that are holomorphic in a neighborhood of $\ol{\bD}$.
\end{cor}

\begin{proof}
  Let $r \in (0,1)$ and let $f(z) = \frac{r-z}{1 - rz}$. Then $f \in \Aut(\bD)$ and in particular
  $\|f\|_\infty \le 1$. On the other hand,
  \begin{equation*}
    |f'(r)| = \frac{1}{1 - r^2} \ge \frac{1}{2(1 - r)},
  \end{equation*}
  so Lemma \ref{lem:der_dirichlet} implies that
  \begin{equation*}
    \|f\|_{\Mult(\cD),2} \ge \frac{1}{C} (1 - r) \log\Big( \frac{1}{1 - r}\Big)^{1/2} |f'(r)|
    \ge \frac{1}{2 C} \log\Big( \frac{1}{1 - r}\Big)^{1/2},
  \end{equation*}
  which tends to infinity as $r \to 1$.
\end{proof}

We will show that in spite of the last corollary, there is no $n \in \bN$ so
that the multiplier norm on the Dirichlet space is equivalent to the $n$-point multiplier
norm. To avoid repetition, we postpone the proof to Section \ref{sec:top_sub_disc},
where we establish a more general result about subhomogeneity of the multiplier algebra
of the Dirichlet space.

\section{Completely isometric subhomogeneity}
\label{sec:ci_subhom}

The goal of this section is to prove Theorems \ref{thm:ci_n_point_intro} and \ref{thm:ci_subhom_intro}. To
this end, we will make use of Arveson's theory of boundary representations \cite{Arveson69}
(see also \cite[Chapter 4]{BL04} and \cite{DK15}),
which generalizes the classical notion of the Choquet boundary of a uniform algebra.

Let $\cA \subset \cB(\cH)$ be a unital subalgebra. It follows from Arveson's extension
theorem that every unital completely contractive map $\varphi: \cA \to \cB(\cK)$ extends to a unital
completely positive map $\pi: C^*(\cA) \to \cB(\cK)$.
The map $\varphi$ is said to have the \emph{unique extension property} if there is a unique such
extension, and this extension is a $*$-homomorphism. If, in addition, the extension $\pi$ is
an irreducible representation of $C^*(\cA)$, then it is called a \emph{boundary representation}
of $\cA$. This notion is in fact independent of the concrete representation of $\cA$
as an operator algebra \cite[Theorem 2.1.2]{Arveson69}.

In general, completely isometric copies of an operator algebra $\cA$ can generate many different $C^ *$-algebras.
The \emph{$C^*$-envelope} is the smallest one in the following sense.
A \emph{$C^*$-cover}
of $\cA$ is a $C^*$-algebra $\kA$ together with a unital complete isometry $j: \cA \to \kA$
so that $\kA = C^*( j(\cA))$. The $C^*$-envelope is a $C^*$-cover $\iota: \cA \to C^*_{env}(\cA)$
so that for every other $C^*$-cover $j: \cA \to \kA$, there is a $*$-homomorphism $\pi: \kA \to C^*_{env}(\cA)$ so that $\iota = \pi \circ j$.
As an example, completely isometric copies of the disc algebra $A(\bD)$ generate
the $C^*$-algebras $C(\bT), C(\ol{\bD})$ and the Toeplitz algebra. The $C^*$-envelope
of the disc algebra is $C(\bT)$.

It is a theorem of Davidson and Kennedy \cite{DK15}, proved
by Arveson in the separable case \cite{Arveson08}, that every operator algebra
has sufficiently many boundary representations in the sense that their direct sum $\pi$ is completely
isometric. As a consequence, the $C^*$-envelope of $\cA$ is the $C^*$-algebra
generated by $\pi(\cA)$. Boundary representations and $C^*$-envelopes have a long and rich
history, which we will not review here. Instead, we refer to \cite{Arveson08,DK15} and the references therein.

Recall from the introduction that a unital $C^*$-algebra $\kA$ is called $n$-subhomogeneous
if every irreducible representation of $\kA$ has dimension at most $n$. Equivalently, for all
$a \in \kA$,
\begin{equation}
  \label{eqn:subhom_c_star}
  \|a\| = \sup \{ \|\pi(a) \| : \pi: \kA \to M_k \text{ is a unital $*$-homomorphism, } k \le n \}.
\end{equation}
(For each $a\in \mathcal{A}$, there is an irreducible GNS-representation $\pi$ of $\mathcal{A}$ with $\|\pi(a)\| = \|a\|$, so $n$-subhomogeneity
implies \eqref{eqn:subhom_c_star}. Conversely, if \eqref{eqn:subhom_c_star} holds for all $a \in \mathcal{A}$,
then $\mathcal{A}$ embeds into a product $\prod_{i \in I} M_{n_i}$ with $n_i \le n$ for all $i$, which implies
that $\mathcal{A}$ is $n$-subhomogeneous; see \cite[Proposition IV.1.4.6]{Blackadar06}.)
Recall further that we defined a unital operator algebra $\mathcal{A}$ to be
completely isometrically $n$-subhomogeneous if for all $r \in \mathbb{N}$
and all $A \in M_r(\mathcal{A})$,
\begin{equation*}
  \|A\| = \sup \{ \|\pi^{(r)}(A) \|: \pi: \mathcal{A} \to M_k \text{ is a u.c.c.\ homomorphism, } k \le n \}.
\end{equation*}

The following proposition connects completely isometric subhomogeneity of operator algebras
to subhomogeneity of $C^*$-algebras.

\begin{prop}
  \label{prop:boundary_subhom}
  Let $\cA \subset \cB(\cH)$ be a unital operator algebra and let $n \in \bN$.
  The following assertions are equivalent:
  \begin{enumerate}[label=\normalfont{(\roman*)}]
    \item $\cA$ is completely isometrically $n$-subhomogeneous,
    \item the $C^*$-envelope of $\cA$ is an $n$-subhomogeneous $C^*$-algebra,
    \item every boundary representation of $\cA$ acts on a Hilbert space of dimension at most $n$.
  \end{enumerate}
\end{prop}

\begin{proof}
  (i) $\Rightarrow$ (ii)
  Suppose that $\cA$ is completely isometrically $n$-subhomogeneous. Then
  there exists an index set $I$ and natural numbers $(n_i)_{i \in I}$ with $\sup_{i \in I} n_i \le n$
  and a unital complete isometry
  \begin{equation*}
    \Phi: \cA \to \prod_{i \in I} M_{n_i}.
  \end{equation*}
  The second characterization of $n$-subhomogeneity mentioned in the discussion
  before the proposition shows that the product on the right and hence also the subalgebra $C^*(\Phi(\cA))$ are $n$-subhomogeneous.
  Since the $C^*$-envelope of $\cA$ is a quotient of $C^*(\Phi(\cA))$,
  the first characterization of $n$-subhomogeneity shows that the $C^*$-envelope is $n$-subhomogeneous as well.
  (These permanence properties for $n$-subhomogeneous $C^*$-algebras also follow directly from \cite[Proposition IV.1.4.6]{Blackadar06}.)

  (ii) $\Rightarrow$ (iii)
  The invariance principle
  for boundary representations \cite[Theorem 2.1.2]{Arveson69} (see \cite[Proposition 3.1]{Arveson} for a
  modern proof) shows that every boundary representation of $\cA$ is an irreducible
  representation of the $C^*$-envelope, hence it acts on a Hilbert space of dimension at most $n$.

  (iii) $\Rightarrow$ (i) The main result of \cite{DK15} shows that the direct
  sum of all boundary representations of $\cA$ is completely isometric on $\cA$,
  hence $\cA$ is completely isometrically $n$-subhomogeneous.
\end{proof}

Recall that a regular unitarily invariant space is a reproducing kernel Hilbert space on $\bB_d$ with reproducing
kernel of the form
\begin{equation*}
  K(z,w) = \sum_{n=0}^\infty a_n \langle z,w \rangle^n,
\end{equation*}
where $a_0 = 1$, $a_n > 0$ for all $n \in \bN$ and $\lim_{n \to \infty} \frac{a_n}{a_{n+1}} = 1$.
If $\cH$ is regular unitarily invariant space, then the polynomials are multipliers of $\cH$.
We can now prove Theorem \ref{thm:ci_subhom_intro},
which we restate for the convenience of the reader. As explained in the introduction, this will establish
Theorem \ref{thm:ci_n_point_intro} as well.

\begin{thm}
  \label{thm:ci_subhom}
  Let $\cH$ be a regular unitarily invariant space on $\bB_d$.
  Then $\Mult(\cH)$ is completely isometrically subhomogeneous if and only if $\Mult(\cH) = H^\infty(\bB_d)$
  completely isometrically.
\end{thm}

\begin{proof}
  It is clear that $H^\infty(\bB_d)$ is completely isometrically $1$-subhomogeneous.

  Conversely, let $A(\cH)$ denote the norm closure of the polynomials inside of $\Mult(\cH)$
  and let $A(\bB_d)$ denote the ball algebra.
  Then the natural inclusion $A(\cH) \hookrightarrow A(\bB_d)$ is completely contractive.
  We first show that if $A(\cH)$ is completely isometrically subhomogeneous, then this
  inclusion is a complete isometry.

  To this end, suppose that the inclusion $A(\cH) \hookrightarrow A(\bB_d)$ is not a complete isometry.
  The maximum modulus principle then shows
  that the map
  \begin{equation*}
    R: A(\cH) \to C(\partial \bB_d), \quad f \mapsto f \big|_{\partial \bB_d},
  \end{equation*}
  is not a complete isometry.
  Since $\cH$ is regular, \cite[Theorem 4.6]{GHX04} yields a short exact sequence of $C^*$-algebras
  \begin{equation*}
    0 \to \cK(\cH) \to C^*(A(\cH)) \to C( \partial \bB_d) \to 0,
  \end{equation*}
  where the first map is the inclusion map and the second map agrees
  with $R$ on $A(\cH)$. Thus, the quotient map by the compact operators on $C^*(A(\cH))$
  is not completely isometric on $A(\cH)$. In this setting, Arveson's boundary
  theorem \cite[Theorem 2.1.1]{Arveson72} implies that the identity representation of $C^*(A(\cH))$ is a boundary
  representation for $A(\cH)$. In particular, $A(\cH)$ has an infinite dimensional boundary
  representation and is therefore not completely isometrically subhomogeneous by Proposition \ref{prop:boundary_subhom}.

  Finally, we turn from $A(\cH)$ to $\Mult(\cH)$.
  The natural inclusion $\Mult(\cH) \hookrightarrow H^\infty(\bB_d)$ is completely contractive.
  If $\Mult(\cH)$ is completely isometrically subhomogeneous,
  then so is the subalgebra $A(\cH)$. By the preceding paragraph, $A(\cH) = A(\bB_d)$ completely isometrically.
  If $F$ belongs the unit ball of $M_n(H^\infty(\bB_d))$, then for each $r \in (0,1)$,
  the function $F_r(z) = F(r z)$ belongs to the unit ball of $M_n(A(\bB_d))$,
  hence $F_r$ belongs the to the unit ball of $M_n(\Mult(\cH))$ for all $r \in (0,1)$.
  Since $F_r$ converges to $F$ pointwise as $r \to 1$, we conclude that $F$ belongs to the unit ball
  of $M_n(\Mult(\cH))$. Therefore, $\Mult(\cH) = H^\infty(\bB_d)$ completely isometrically.
\end{proof}

The question of when the identity representation is a boundary representation of $A(\cH)$
was already studied in \cite{GHX04}.

We record two concrete applications to spaces on the unit ball.
Recall that $\cD_a(\bB_d)$ is the reproducing kernel Hilbert space on $\bB_d$ with kernel $K(z,w)
= \frac{1}{(1 - \langle z,w \rangle)^a}$, and $\cD_0(\bB_d)$ has kernel
$K(z,w) = \frac{1}{\langle z,w \rangle} \log( \frac{1}{1 - \langle z,w \rangle})$.

\begin{cor}
  \label{cor:ci_subhom_ball_dirichlet}
  Let $a \in [0,\infty)$. Then $\Mult(\cD_a(\bB_d))$ is completely isometrically subhomogeneous if and only if $a \ge d$.
\end{cor}

\begin{proof}
  Let $\cH = \cD_a(\bB_d)$.
  If $a = d$, then $\cH$ is the Hardy space on the ball; if $a > d$, then $\cH$ is a weighted Bergman space;
  see, for instance, \cite[Theorem 2.7 and Proposition 4.28]{Zhu05}.
  In either case, $\Mult(\cH) = H^\infty(\bB_d)$,
  which is completely isometrically $1$-subhomogeneous.

  Conversely, one can deduce from Examples 1 and 2 in \cite{GHX04} that the identity
  representation of $A(\cH)$, the norm closure of the polynomials in $\Mult(\mathcal{H})$, is a boundary representation if $a < d$, so that $\Mult(\cH)$ is not completely
  isometrically subhomogeneous by Proposition \ref{prop:boundary_subhom}.
  
  Alternatively,
  we can argue with Theorem \ref{thm:ci_subhom} as follows.
  If $a > 0$, then 
  \begin{equation*}
    \|z_i\|_{\mathcal{H}}^2 = \frac{\Gamma(a)}{\Gamma(a+1)} = \frac{1}{a}.
  \end{equation*}
  If $\Mult(\cH)$ is completely isometrically subhomogeneous, then
  Theorem \ref{thm:ci_subhom} implies that $\Mult(\cH) = H^\infty(\bB_d)$ completely isometrically.
  In particular,
  \begin{equation*}
    1 =
    \left\|
    \begin{bmatrix}
      z_1 \\ \vdots \\ z_d
    \end{bmatrix} \right\|_{\Mult(\cH,\cH \otimes \bC^d)}
    \ge \sum_{n=1}^d \|z_i\|^2_{\cH} = \frac{d}{a},
  \end{equation*}
  so that $a \ge d$. Similarly, if $a=0$, then $\|z_i\|^2_{\cH} = 2$, so that $\Mult(\cH)$ is not completely isometrically subhomogeneous
  by the same reasoning.
\end{proof}

Our second application concerns spaces with the complete
Nevanlinna--Pick property; see \cite{AM02} for background on this topic.

\begin{cor}
  Let $\cH$ be a regular unitarily invariant space on $\bB_d$ with the complete
  Nevanlinna--Pick property. Then $\Mult(\cH)$ is completely isometrically subhomogeneous if and only if $d=1$
  and $\cH = H^2(\bD)$.
\end{cor}

\begin{proof}
  It is clear that $\Mult(H^2(\bD)) = H^\infty(\bD)$ is completely isometrically $1$-sub\-homogeneous.
  Conversely, it was proved in \cite[Theorem 6.2]{CH18} that the identity representation of $A(\cH)$ is a boundary
  representation unless $\cH = H^2(\bD)$, so the result follows from Proposition \ref{prop:boundary_subhom}.

  Alternatively, we can again argue with Theorem \ref{thm:ci_subhom} directly.
  Observe that if $K(z,w) = \sum_{n=0}^\infty a_n \langle z,w \rangle^n$, then $\|z_i\|^2_{\cH} = \frac{1}{a_1}$.
  If $\Mult(\cH)$ is completely
  isometrically subhomogeneous, then $\Mult(\cH) = H^\infty(\bB_d)$ completely isometrically by Theorem \ref{thm:ci_subhom}, so
  \begin{equation*}
    1 = \left\|
    \begin{bmatrix}
      z_1 \\ \vdots \\ z_d
    \end{bmatrix} \right\|_{\Mult(\cH,\cH \otimes \bC^d)}
    \ge \sum_{n=1}^d \|z_i\|^2_{\cH} = \frac{d}{a_1},
  \end{equation*}
  hence $a_1 \ge d$.
  Since $\cH$ has the complete Nevanlinna--Pick property,
  there exists a sequence $(b_n)$ of non-negative numbers with $\sum_{n=1}^\infty b_n \le 1$ and
  $\sum_{n=0}^\infty a_n t^n = \frac{1}{1 - \sum_{n=1}^\infty b_n t^n}$ for $t \in \bD$. (This is a variant of \cite[Theorem 7.33]{AM02}, see for instance \cite[Lemma 2.3]{Hartz17a} for the precise statement.)
  Since $a_1 = b_1$,
  it follows that $b_1 = d =1$ and $b_n = 0$ for $n \ge 2$, hence $\cH = H^2(\bD)$.
\end{proof}

\section{Topological subhomogeneity for spaces on the disc}
\label{sec:top_sub_disc}

In this section, we study topological subhomogeneity of multiplier algebras
of weighted Dirichlet spaces on the unit disc. As in Section \ref{sec:n-point},
the basic idea is to interpolate holomorphic functions in $\bD$ by finite Blaschke products.

If $A$ is a diagonalizable $n \times n$ matrix with $\sigma(A) \subset \bD$, then
classical Nevanlinna--Pick interpolation shows that for any $f \in H^\infty(\bD)$ with $\|f\|_\infty \le 1$,
there exists a finite Blaschke product $B$ of degree at most $n-1$ and $\lambda \in \ol{\bD}$
such that $f$ and $\lambda B$ agree on $\sigma(A)$, hence $f(A) = \lambda B(A)$.
We require a generalization of this fact to not necessarily diagonalizable matrices,
which corresponds to interpolating suitable derivatives as well.
This generalization readily follows from Sarason's approach to Nevanlinna--Pick interpolation \cite{Sarason67}.

\begin{lem}
  \label{lem:Blaschke_product}
  Let $A \in M_n$ with $\sigma(A) \subset \bD$
  and let $f \in H^\infty$ with $||f||_\infty \le 1$.
  Then there exists a finite Blaschke product $B$ of degree at most $n-1$ and a complex
  number $\lambda$ with $|\lambda| \le 1$ such that $f(A) = \lambda B(A)$.
\end{lem}

\begin{proof}
  Clearly, we may assume that $f \neq 0$ and that $A \neq 0$.
  Let $\psi$ be the finite Blaschke product of degree at most $n$
  whose zeros are those of the minimal polynomial
  of $A$, counted with multiplicity.
  Let $K = H^2 \ominus \psi H^2$.
  The corollary to Proposition 5.1 of \cite{Sarason67} shows that
  there exists a unique function $\varphi \in H^\infty$ with
  \begin{equation}
    \label{eqn:sarason}
 P_K M_\varphi \big|_K =
    P_K M_f \big|_{K}
    \quad \text{ and } \quad
    \|\varphi\|_\infty = \|P_K M_f \big|_{K} \|;
  \end{equation}
  moreover, this $\varphi$ is a rational function of constant modulus
  on the unit circle with strictly fewer zeros than $\psi$. In other words,
  $\varphi = \lambda B$ for a finite Blaschke product $B$ of degree at most $n-1$
  and a number $\lambda \in \bC$ with $|\lambda| \le 1$.
  Equation \eqref{eqn:sarason} and co-invariance of $K$ under multiplication
  operators imply that $f - \lambda B \in \psi H^2$.
  Consequently, $f(A) = \lambda B(A)$, as desired.
\end{proof}

The next step is to establish a version of Lemma \ref{lem:n_point}
for representations of multiplier algebras. In order to be able to treat
the classical Dirichlet space as well, we formulate and prove a more flexible version
in which $\Aut(\bD)$ is not assumed to be a bounded subset of the multiplier algebra.

\begin{lem}
  \label{lem:subhomogeneous}
  Let $\cH$ be a reproducing kernel Hilbert space on $\bD$
  and suppose that $\Aut(\bD) \subset \Mult(\cH)$. For $r \in [0,1)$, let
  \begin{equation*}
    h(r) = \sup_{\theta \in \Aut(\bD)} \| \theta(r z) \|_{\Mult(\cH)}.
  \end{equation*}
  Let $\pi: \Mult(\cH) \to M_n$ be a unital bounded homomorphism.
  Then for all $r \in (0,1)$, the inclusion $H^\infty( r^{-1} \bD) \subset \Mult(\cH)$
  holds, and
  \begin{equation*}
    \|\pi(f)\| \le \|\pi\| h(r)^{n-1} \sup_{z \in r^{-1} \bD} |f(z)|
  \end{equation*}
  for all $f \in H^\infty(r^{-1} \bD)$.
\end{lem}

\begin{proof}
  We first observe that the condition $\Aut(\bD) \subset \Mult(\cH)$ implies
  that $z \in \Mult(\cH)$ and that $\sigma(z) \subset \ol{\bD}$.
  Indeed, if $\lambda \notin \ol{\bD}$, then
  \begin{equation*}
    \frac{\ol{\lambda}^{-1} - z}{1 - \lambda^{-1} z} \in \Mult(\cH),
  \end{equation*}
  hence $\ol{\lambda}^{-1} - z$ belongs to the ideal of $\Mult(\cH)$
  generated by  $\lambda - z$, and hence $1$ belongs to this ideal, so that $\lambda -z$
  is invertible. The analytic functional calculus therefore implies that every function
  analytic in a neighborhood of $\ol{\bD}$ is a multiplier of $\cH$.
  
  Next, let $\pi: \Mult(\cH) \to M_n$ be a unital bounded homomorphism, let $r \in (0,1)$
  and let $f$ be an element of the unit ball of $H^\infty(r^{-1} \bD)$.
  Define $A = \pi(z) \in M_n$, so that $\sigma(T) \subset \ol{\bD}$ by the first paragraph.
  Applying Lemma \ref{lem:Blaschke_product} to the matrix $r A$ and the function $z \mapsto f(r^{-1} z)$,
  we find $\lambda \in \ol{\bD}$ and a finite Blaschke product $B$ of degree at most $n-1$ so that
  $\lambda B(r A) = f(A)$. Using that $\pi$ is a unital bounded homomorphism, we conclude that
  \begin{equation*}
    \pi(f) = f(A) = \lambda B(r A) = \lambda \pi( B(r z)).
  \end{equation*}
  Since $B$ is a product of at most $n-1$ disc automorphisms, $\|B (r z)\|_{\Mult(\cH)} \le h(r)^{n-1}$,
  so that
  \begin{equation*}
    \|\pi(f)\| \le \|\pi\| \|B( r z) \|_{\Mult(\cH)} \le \|\pi\| h(r)^{n-1},
  \end{equation*}
  which is the desired estimate.
\end{proof}

The following result is a generalization of Corollary \ref{cor:many_points}.

\begin{cor}
  \label{cor:not_subhom}
  Let $a \in (0,1)$. Then $\Mult(\cD_a)$ is not topologically subhomogeneous.
\end{cor}

\begin{proof}
  As explained in the proof of Corollary \ref{cor:many_points}, $\Aut(\bD) \subset \Mult(\cD_a)$
  and
  \begin{equation*}
    \|\theta\|_{\Mult(\cD_a)} = \|z\|_{\Mult(\cD_a)}
  \end{equation*}
  for all $\theta \in \Mult(\cD_a)$. Moreover, it is well known that if $\varphi \in \Mult(\cD_a)$
  and $r \in (0,1)$, then $\varphi(r z) \in \Mult(\cD_a)$ and $\|\varphi(r z) \|_{\Mult(\cD_a)}
  \le \|\varphi\|_{\Mult(\cD_a)}$; this follows, for instance, from rotational invariance of $\Mult(\cD_a)$
  and a routine application of the Poisson kernel. Thus, the function $h$ in Lemma \ref{lem:subhomogeneous}
  is bounded above by $\|z\|_{\Mult(\cD_a)}$. An application of Lemma \ref{lem:subhomogeneous} therefore
  shows that for every unital bounded homomorphism $\pi: \Mult(\cH) \to M_n$
  and every polynomial $f \in \bC[z]$, the estimate
  \begin{equation*}
    \|\pi(f)\| \le \|\pi\| \|z\|^{n-1}_{\Mult(\cD_a)} \sup_{z \in \ol{\bD}} |f(z)|
  \end{equation*}
  holds. Since the multiplier norm on $\cD_a$ is not dominated by a constant times the supremum norm
  on $\ol{\bD}$ (see the proof of Corollary \ref{cor:many_points}),
  it follows that $\Mult(\cD_a)$ is not topologically subhomogeneous.
\end{proof}

To apply Lemma \ref{lem:subhomogeneous} in the case of the Dirichlet space,
we require the following estimate.

\begin{lem}
  \label{lem:dirichlet_auto_multiplier}
  There exists a constant $C > 0$ so that for all $\theta \in \Aut(\bD)$
  and all $r \in (0,1)$,
  \begin{equation*}
    \|\theta(r z) \|_{\Mult(\cD)} \le C \log \Big( \frac{2}{1-r} \Big)^{1/2}.
  \end{equation*}
\end{lem}

\begin{proof}
  It is a result of Brown and Shields \cite[Proposition 18]{BS84} that if $f \in H^\infty$ and
  $\sum_{n=1}^\infty (n \log(n)) | \widehat f(n)|^2 < \infty$, then $f \in \Mult(\cD)$; see also \cite[Exercise 5.1.3]{EKM+14}. Here $f(z) = \sum_{n=0}^\infty \widehat f(n) z^n$.
  The closed graph theorem (or an inspection of the proof of this result) then shows that
  there exists a constant $C > 0$ such that
  \begin{equation*}
    \|f\|_{\Mult(\cD)} \le C \Big(  \|f\|_\infty + \Big( \sum_{n=1}^\infty n \log(n) | \widehat f(n)|^2 \Big)^{1/2} \Big)
  \end{equation*}
  for all functions $f$ that are holomorphic in a neighborhood of $\ol{\bD}$.
  We will apply this inequality to a disc automorphism $\theta(z) = \frac{a  - z}{1 - \ol{a} z}$, where $a \in \bD$,
  and the function $f(z) = \theta(r z)$. Note that
  \begin{equation*}
  f(z) =  a - \sum_{n=1}^\infty r (\ol{a} r)^{n-1} (1 - |a|^2) z^n.
  \end{equation*}

  It therefore suffices to show that there exists a (possibly different) constant $C > 0$
  so that for all $a \in \bD$ and all $r \in [0,1)$,
  \begin{equation}
    \label{eqn:auto_mult}
    \sum_{n=1}^\infty n \log(n) r^2 |a r|^{2 n -2} (1 - |a|^2)^2 \le
    C \log \Big( \frac{2}{1 - r} \Big).
  \end{equation}
  Clearly, we may assume that $|a| \ge \frac{1}{2}$ and $r \ge \frac{1}{2}$.
  We use the known asymptotic
  \begin{equation*}
    \sum_{n=1}^\infty n \log(n) x^n \sim \frac{1}{(1 - x)^2} \log \Big( \frac{1}{1- x} \Big) \quad
    \text{ as } x \nearrow 1,
  \end{equation*}
  which for instance can be deduced from \cite[Chapter VII, Example 5, p. 242]{Titchmarsh52}.
  With this asymptotic identity, we estimate the left hand side of \eqref{eqn:auto_mult}
  as
  \begin{align*}
    r^2 (1 - |a|^2)^2 \sum_{n=1}^\infty n \log(n) |a r|^{2 n -2}
  &\lesssim \frac{ (1 - |a|^2)^2}{( 1- | a r|^2)^2} \log \Big( \frac{1}{1 - |a r|^2} \Big) \\
  &\le \log \Big( \frac{1}{1 - r} \Big), 
  \end{align*}
  where the implied constants do not depend on $a$ or on $r$. This estimate finishes the proof.
\end{proof}

\begin{rem}
  Lemma \ref{lem:der_dirichlet} shows that the estimate in Lemma \ref{lem:dirichlet_auto_multiplier}
  is best possible up to constants. Indeed, if $\theta(z) = - \frac{r - z}{1 - r z}$
  and $f(z) = \theta(rz)$, then $f'(r) = \frac{r (1-r^2)}{(1-r^3)^2} \sim \frac{2}{9(1-r)}$ as $r \nearrow 1$,
  so
  \begin{equation*}
  \|\theta(r z) \|_{\Mult(\cD)} \gtrsim \log \Big( \frac{1}{1 - r} \Big)^{1/2}
    \quad \text{ as } r \nearrow 1.
  \end{equation*}
\end{rem}

We are now ready to show that $\Mult(\cD)$ is not subhomogeneous.

\begin{prop}
  \label{prop:dirichlet_subhom}
  The algebra $\Mult(\cD)$ is not topologically subhomogeneous.
  In particular, there does not exist a constant $C > 0$ and $n \in \bN$ so that
  \begin{equation*}
    \|\varphi\|_{\Mult(\cD)} \le C \|\varphi\|_{\Mult(\cD),n}
  \end{equation*}
  for all $\varphi \in \Mult(\cD)$.
\end{prop}

\begin{proof}
  Assume towards a contradiction that there exist $n \in \bN$ and constants $C_1,C_2 > 0$
  so that for all $f \in \bC[z]$, the estimate
  \begin{equation}
    \label{eqn:subhom_contra}
    \|f\|_{\Mult(\cD)} \le C_1 \sup\{ \|\pi(f) \| \}
  \end{equation}
  holds,
  where the supremum is taken over all unital homomorphisms $\pi: \Mult(\cD) \to M_k$
  with $\|\pi\| \le C_2$ and $k \le n$.

  Lemma \ref{lem:subhomogeneous} and
  Lemma \ref{lem:dirichlet_auto_multiplier} show that there
  exists a constant $C > 0$ so that for all $r \in (0,1)$,  all $f \in \bC[z]$
  and all representations $\pi: \Mult(\cD) \to M_k$ as above,
  \begin{equation*}
    \|\pi(f)\|
    \le C_2 C^{n-1} 
    \log \Big( \frac{2}{1 -r} \Big)^{(n-1)/2}
    \sup_{|z| \le r^{-1}} |f(z)|.
  \end{equation*}
  We apply this inequality to $f(z) = z^m$ for $m$ large and $r = 1 - \frac{1}{m}$.
  Since $(1 - \frac{1}{m})^{-m}$ tends to $e$, we find that
  \begin{equation*}
    \|\pi(z^m)\| \lesssim \log(2 m)^{(n-1)/2},
  \end{equation*}
  where the implied constant is independent of $m$ and of the particular
  representation $\pi$. On the other hand,
  \begin{equation*}
    \|z^m\|_{\Mult(\cD)} \ge \|z^m\|_{\cD} = \sqrt{m+1}.
  \end{equation*}
  Since 
  \begin{equation*}
    \frac{\sqrt{m+1}}{\log(2 m)^{(n-1)/2}} \to \infty
  \end{equation*}
  as $m \to \infty$, this contradicts
  the assumption \eqref{eqn:subhom_contra}.
  Finally, the additional statement follows from the fact that the $n$-point
  norm is dominated by the supremum on the right-hand side of \eqref{eqn:subhom_contra};
  see Equation \eqref{eqn:subhom} in the introduction.
\end{proof}

\section{Topological subhomogeneity for spaces on the ball}
\label{sec:subhom_ball}

In this section, we generalize the results about topological subhomogeneity in the preceding section from the unit disc
to the unit ball. To deduce the higher dimensional statements from the ones in dimension one, we establish
an embedding result. To this end, we require variant of a result of Kacnelson \cite{Kacnelson72}.
Let $I$ be a totally ordered set and let $\cH$ be a Hilbert space with orthogonal basis $(e_i)_{i \in I}$.
We say that a bounded linear operator $T$ on $\cH$ is lower triangular if
\begin{equation*}
  \langle T e_i, e_j \rangle = 0 \quad \text{ whenever } j < i,
\end{equation*}
and we write $\cT(\cH)$ for the algebra of all bounded lower triangular operators on $\cH$.
The statement below can be found
in \cite[Corollary 3.2]{AHM+18a} in the case when $I = \bN_0$, equipped with
the usual ordering.

\begin{lem}
  \label{lem:kacnelson_triangular}
  Let $\cH$ and $\cK$ be Hilbert spaces with $\cH \subset \cK$ as vector spaces.
  Let $I$ be a totally ordered set and let $(e_i)_{i \in I}$ be a sequence
  of vectors in $\cH$ that is an orthogonal basis both for $\cH$ and for $\cK$.
  If the family $( \|e_i\|_{\cH} / \|e_i\|_{\cK})$
  is non-decreasing with respect to the order on $I$,
  then $\cT(\cH) \subset \cT(\cK)$, and the inclusion is a complete contraction.
\end{lem}

\begin{proof}
  If $I$ is a finite set, the result was established in \cite[Corollary 3.2]{AHM+18a}.
  We will reduce the general case to this particular case. To this end,
  let $\cF(I)$ denote the set of finite subsets of $I$, which is a directed
  set under inclusion.
  For $J \in \cF(I)$,
  let $\cH_J \subset \cH$ and $\cK_J \subset \cK$ denote the linear spans of $\{e_i: i \in J\}$,
  respectively, so that $\cH_J = \cK_J$ as vector spaces, but with possibly different norms.
  Moreover, let $P_J \in B(\cK)$ denote the orthogonal projection onto
  $\cK_J$, and notice that on $\cH$, the operator $P_J$
  also acts as the orthogonal projection onto $\cH_J$.
  Since for any $T \in \cT(\cH)$, the operator
  $
    P_J T P_J
  $
  is lower triangular on $\cH_J$, the case of finite $I$ shows that
  $\|P_J T P_J\|_{B(\mathcal{K})} \le \|T\|_{B(\mathcal{H})}$, and
  the net $(P_J T P_J)_{J \in \cF(I)}$  converges to $T$ in the strong operator topology of $\cB(\cK)$.
  The general result follows from this observation.
\end{proof}

The main tool of this section is the following lemma. It allows us to
embed multiplier algebras on the unit disc into multiplier
algebras on the unit ball by extending a function $f$ on $\bD$
to the function $z \mapsto f(z_1)$ on the unit ball.

\begin{lem}
  \label{lem:first_proj}
  Let $\cH$ be a reproducing kernel Hilbert space on $\bB_d$ with reproducing
  kernel
  \begin{equation*}
    K(z,w) = \sum_{n=0}^\infty a_n \langle z,w \rangle^n,
  \end{equation*}
  where $a_n > 0$ for all $n \in \bN_0$.
  Let $d' \le d$ and let $\cH'$ be the reproducing kernel Hilbert space on $\bB_{d'}$ with
  reproducing kernel $K \big|_{ \bB_{d'} \times \bB_{d'}}$.
  Then the map
  \begin{equation*}
    \Phi: \Mult(\cH') \to \Mult(\cH), \quad \varphi \mapsto \varphi \circ P,
  \end{equation*}
  where $P$ denotes the projection from $\bC^d$ onto $\bC^{d'}$ given
  by $P(z_1,\ldots,z_d) = (z_1,\ldots,z_{d'})$,
  is a complete isometry.
\end{lem}

\begin{proof}
We will repeatedly use the basic fact that the monomials $(z^\alpha)_{\alpha \in \mathbb{N}_0^d}$
  form an orthogonal basis of $\mathcal{H}$ with
  \begin{equation*}
    \|z^\alpha\|^2_{\mathcal{H}} = \frac{\alpha!}{a_n |\alpha|!}
  \end{equation*}
  In particular, the map
  \begin{equation*}
    V: \cH' \to \cH, \quad f \mapsto f \circ P,
  \end{equation*}
  is an isometry, so it suffices to show that $\Phi$ is a complete contraction.

  Clearly, we may assume that $d' < d$.
  Throughout the proof, we write $z= (u,v)$, where $u = (z_1,\ldots,z_{d'})$
  and $v = (z_{d'+1} ,\ldots, z_d)$.
  For $\beta \in \bN_0^{d - d'}$, let
  \begin{equation*}
    H_{\beta} = \bigvee \{u^\alpha v^\beta : \alpha \in \bN_0^{d'}\} \subset \cH,
  \end{equation*}
  so that $\cH$ is the orthogonal direct sum of the subspaces $H_{\beta}$. Moreover,
  $H_0 = V \cH'$.
  Observe that if $\varphi \in \Mult(\cH')$, then the (a priori unbounded) operator
  $M_{\varphi \circ P}$ preserves this direct sum decomposition. Thus, it suffices
  to show that for all $\beta \in \bN_0^{d - d'}$, $n \in \bN_0$ and all $[\varphi_{i j}] \in M_n(\Mult(\cH))$,
  \begin{equation*}
    \| [M_{\varphi_{i j} \circ P} \big|_{H_\beta}] \|_{M_n(B(H_\beta))}
    \le \| [M_{\varphi_{i j}}] \|_{M_n(B(\cH'))}.
  \end{equation*}

  To this end, let $\beta \in \bN_0^{d - d'}$, let $k = |\beta|$
  and let $\widetilde H_{\beta}$ be the Hilbert space of formal power series
  in the variable $u = (z_1,\ldots,z_{d'})$ with orthogonal basis $(u^{\alpha} z_1^k)_{\alpha \in \bN_0^{d'}}$ and norm
  defined by
  \begin{equation*}
    \| u^\alpha z_1^k \|_{\widetilde H_{\beta}} = \|u^\alpha v^\beta \|_{\cH}.
  \end{equation*}
  If $\varphi \in \Mult(\cH')$, then $M_{\varphi \circ P} \big|_{H_{\beta}}$ is unitarily equivalent
  to $M_{\varphi}$ on $\widetilde H_{\beta}$.
  To compare the norm of $M_\varphi$ on $\widetilde H_\beta$ and on $\cH'$, we will use
  Lemma \ref{lem:kacnelson_triangular}.
  Let $\cH'_k = \bigvee \{u^\alpha z_1^k : \alpha \in \bN_0^{d'} \} \subset \cH'$. Then
  \begin{equation}
    \label{eqn:beta}
    \frac{ \| u^\alpha z_1^k \|^2_{\cH_k'} }{ \|u^\alpha z_1^k \|_{\widetilde \cH_{\beta}}^2} =
    \frac{ \|u^\alpha z_1^k \|^2_{\cH}}{\|u^\alpha v^\beta \|^2_{\cH}}
    = \frac{ (\alpha + k e_1 )!}{\alpha! \beta!}
    = \frac{ \prod_{j=1}^k (\alpha_1 + j)}{\beta!},
  \end{equation}
  which is increasing in $\alpha_1$. Thus, $\cH_k' \subset \widetilde \cH_\beta$ as vector spaces, and if we order $\bN_0^{d'}$ lexicographically,
  then the quantity in \eqref{eqn:beta} is non-decreasing in $\alpha \in \bN_0^{d'}$. Moreover,
  if $\varphi \in \Mult(\cH')$, then the operator $M_{\varphi}$ is lower triangular with respect
  to the orthogonal basis $(u^\alpha z_1^k)_{\alpha}$ of $\cH_k'$ in this ordering,
  so that by Lemma \ref{lem:kacnelson_triangular},
  \begin{align*}
    \| [M_{\varphi_{i j} \circ P} \big|_{H_\beta}] \|_{M_n(B(H_\beta))}
    = \| [M_{\varphi_{i j}} ] \|_{M_n(B(\widetilde H_\beta))}
    &\le \| [M_{\varphi_{i j}} \big|_{\cH'_k}] \|_{M_n(B(\cH'_k))} \\
    &\le \| [M_{\varphi_{i j}}] \|_{M_n(B(\cH'))}
  \end{align*}
  for all $n \in \bN$ and all $[\varphi_{i j}] \in M_n(\Mult(\cH'))$.
\end{proof}

\begin{rem}
  \label{rem:first_proj_cnp}
  If $\cH$ is a complete Nevanlinna--Pick space, one can argue more easily as follows.
  Let $\Phi \in \Mult(\cH' \otimes \bC^n)$ be a multiplier of norm at most $1$.
  By the complete Nevanlinna--Pick property, there exists a multiplier
  $\Psi \in \Mult(\cH \otimes \bC^n)$ of norm at most $1$ with $\Psi \big|_{\bB_{d'}} = \Phi$.
  Writing $z=(u,v)$ as in the preceding proof, we see that for $(u,v) \in \bB_d$,
  \begin{equation*}
    (\Phi \circ P)(u,v) = \Psi(u,0)
    =  \frac{1}{2 \pi} \int_{0}^{2 \pi} \Psi(u, e^{ i t} v) \, dt.
  \end{equation*}
  The symmetry of $K$ implies that for each $t \in \bR$, the function $(u,v) \mapsto \Phi(u, e^{i t} v)$
  is a multiplier of norm at most $1$, hence $\Phi \circ P$ is a multiplier of norm
  at most $1$ as well. Since the reverse inequality is clear, this argument
  proves the result in the case of a complete Nevanlinna--Pick space.
\end{rem}

We now obtain a negative answer to Question \ref{quest:subhomogeneous} for the spaces $\cD_a(\bB_d)$, where $a \in [0,1)$.
In particular, Question \ref{quest:many_points} has
a negative answer for these spaces as well.
\begin{cor}
  \label{cor:ball_not_subhom}
  Let $a \in [0,1)$ and let $d \in \bN$. Then $\Mult(\cD_a(\bB_d))$ is not topologically subhomogeneous.
  In particular, there do not exist a constant $C > 0$ and $n \in \bN$ so that
  \begin{equation*}
    \|f\|_{\Mult(\cD_a(\bB_d))} \le C
    \|f\|_{\Mult(\cD_a(\bB_d)),n}
  \end{equation*}
  for all $f \in \Mult(\cD_a(\bB_d))$.
\end{cor}

\begin{proof}
  By Proposition \ref{prop:dirichlet_subhom} and Corollary \ref{cor:not_subhom}, the algebra
  $\Mult(\cD_a(\bD))$ is not topologically subhomogeneous. Lemma \ref{lem:first_proj} shows
  that $\Mult(\cD_a(\bD))$ is (completely isometrically) isomorphic to a subalgebra of $\Mult(\cD_a(\bB_d))$,
  hence $\Mult(\cD_a(\bB_d))$ is not topologically subhomogeneous either. As before, the additional
  statement follows from the fact that the $n$-point multiplier norm can be computed using
  representations of dimension at most $n$.
\end{proof}

\section{Embedding weighted Dirichlet spaces into the Drury--Arveson space}
\label{sec:embedding}

In this section, we show that the multiplier algebra of the Drury--Arveson
space $H^2_d$ is not subhomogeneous for $d \ge 2$. This does not follow
from Lemma \ref{lem:first_proj}, as the multiplier algebra of $H^2_1$ is $H^\infty(\bD)$, which is $1$-subhomogeneous.
Instead, we show that multiplier algebras of certain weighted Dirichlet spaces embed completely
isometrically into $\Mult(H^2_d)$.

We begin with a known construction that produces an embedding for the Hilbert function spaces.
Let $d \in \bN$ and let
\begin{equation*}
  \tau: \ol{\bB_d} \to \ol{\bD}, \quad z \mapsto d^{d/2} z_1 z_2 \ldots z_d.
\end{equation*}
It follows from the inequality of arithmetic and geometric means that $\tau$ maps $\ol{\bB_d}$
onto $\ol{\bD}$ and $\bB_d$ onto $\bD$. Let
\begin{equation*}
  a_n = \| \tau^n\|^{-2}_{H^2_d}
\end{equation*}
and let $\cH_d$ be the reproducing kernel Hilbert space on $\bD$ with reproducing kernel
\begin{equation*}
  K(z,w) = \sum_{n=0}^\infty a_n (z \ol{w})^n.
\end{equation*}

The spaces $\cH_d$ are weighted Dirichlet spaces on $\bD$ that embed into $H^2_d$, cf.\ Lemmas 3.1 and 4.1
in \cite{Hartz17}.

\begin{lem}
  \label{lem:H_d_embed}
  Let $d \in \bN$.
  \begin{enumerate}[label=\normalfont{(\alph*)}]
    \item The map
  \begin{equation*}
    V: \cH_d \to H^2_d, \quad f \mapsto f \circ \tau,
  \end{equation*}
  is an isometry whose range consists of all functions in $H^2_d$ that are power series in $z_1 z_2 \ldots z_d$.
  \item The asymptotic identity $a_n \approx (n+1)^{(1-d)/2}$ holds.
  \item The space $\cH_2$ is equal to the weighted Dirichlet space $\cD_{1/2}$, with equality of norms.
  \item The space $\cH_3$ is equal to the classical Dirichlet space $\cD$, with equivalence of norms.
  \item For $d \ge 4$, $\cH_d \subset A(\bD)$ and $\Mult(\cH_d) = \cH_d$ with equivalent norms.
  \end{enumerate}
\end{lem}

\begin{proof}
  (a)
  Since $(z^n)$ and $(\tau^n)$ are orthogonal sequences in $\cH_d$ and $H^2_d$, respectively, $V$
  is an isometry by definition of the sequence $(a_n)$. Moreover, if $M$ is denotes the closed
  subspace of $H^2_d$ generated by the monomials $(z_1 z_2 \ldots z_d)^n$, then the range of $V$ is clearly
  contained in $M$ and contains all polynomials in $z_1 z_2 \cdots z_d$, so it is equal to $M$.

  (b) By Stirling's formula, $n! \sim \sqrt{2 \pi n} (\frac{n}{e})^n$, so
  \begin{equation*}
    a_n^{-1} = \|\tau^n\|^2_{H^2_d} = d^{n d} \frac{ (n!)^d}{(d n)!} \sim
    (2 \pi)^{(d-1)/2} n^{(d-1)/2},
  \end{equation*}
  from which (b) follows.

  (c) follows from a computation involving the binomial series, see \cite[Lemma 4.1]{Hartz17}.

  (d) is a consequence of (b).

  (e) Part (b) shows that the sequence $(a_n)$ belongs to $\ell^1$ for $d \ge 4$, so that $\cH_d$ consists of
  continuous functions on $\ol{\bD}$.
  Finally, the equality $\cH_d = \Mult(\cH_d)$ follows from part (b) combined with
  Proposition 31 and Example 1 in Section 9
  of \cite{Shields74}.
\end{proof}

We will show that the embedding $V$ of Lemma \ref{lem:H_d_embed} is also a complete isometry on the level
of multiplier algebras.

\begin{prop}
  \label{prop:mult_embed}
  For $d \in \bN$, the map
  \begin{equation*}
    \Phi: \Mult(\cH_d) \to \Mult(H^2_d), \quad \varphi \mapsto \varphi \circ \tau,
  \end{equation*}
  is a unital complete isometry.
\end{prop}

We require some preparation for the proof of this result.

\begin{rem}
  It follows from part (c) of Lemma \ref{lem:H_d_embed} that the space $\cH_2$ has the
  complete Nevanlinna--Pick property, and in general, the spaces $\cH_d$ have the complete Nevanlinna--Pick property, at least up to equivalence of norms by part (b) of Lemma \ref{lem:H_d_embed}.
  Indeed, $\mathcal{D}_{1/2}$ is well known to have the complete Nevanlinna--Pick property,
  which can be seen by expanding the reciprocal of the reproducing kernel of $\mathcal{D}_{1/2}$ into a binomial series, see also \cite[Theorem 7.33]{AM02}.
  Moreover, if $\widetilde{\mathcal{H}_d}$ is the reproducing kernel Hilbert space on $\mathbb{D}$ with reproducing kernel
  \begin{equation*}
    \widetilde{K}(z,w) = \sum_{n=0}^\infty (n+1)^{(1-d)/2} (z \overline{w})^n,
  \end{equation*}
  then $\widetilde{\mathcal{H}_d}$ has the complete Nevanlinna--Pick property for all $d \ge 1$ (see, for instance, \cite[Corollary 7.41]{AM02}),
  and Lemma \ref{lem:H_d_embed} (b) implies that $\mathcal{H}_d = \widetilde{\mathcal{H}_d}$ with equivalent norms.
We emphasize that the construction in Proposition \ref{prop:mult_embed} differs from the universality
theorem for complete Nevanlinna--Pick spaces of Agler and \mcc\ \cite{AM00}.

Indeed, the main result of \cite{AM00}
yields an injection $j: \bD \to \bB_d$ for some $d \in \bN \cup \{ \infty \}$
such that
\begin{equation*}
  H^2_d \to \cH_2, \quad f \mapsto f \circ j,
\end{equation*}
is a co-isometry and such that
\begin{equation*}
  \Mult(H^2_d) \to \Mult(\cH_2), \quad \varphi \mapsto \varphi \circ j,
\end{equation*}
is a complete quotient map. In fact, $d= \infty$ is necessary,
see \cite[Corollary 11.9]{Hartz17a}. Thus, \cite{AM00} allows us to realize $\Mult(\cH_2)$ as a quotient
of $\Mult(H^2_\infty)$. In Proposition \ref{prop:mult_embed}, the arrows are reversed.
We obtain a surjection $\tau: \bB_2 \to \bD$ such that
\begin{equation*}
  \cH \to H^2_2, \quad f \mapsto f \circ \tau,
\end{equation*}
is an isometry and such that
\begin{equation*}
  \Mult(\cH) \to \Mult(H^2_2), \quad \varphi \mapsto \varphi \circ \tau,
\end{equation*}
is a complete isometry. Thus, Proposition \ref{prop:mult_embed} allows us to realize $\Mult(\cH)$ as a subalgebra
of $\Mult(H^2_2)$.
\end{rem}

Our proof of Proposition \ref{prop:mult_embed} is related to a proof suggested by Ken Davidson in the case $d=2$.
Davidson's proof uses Schur multipliers and is in turn based on a computation of Jingbo Xia (see \cite[Example 6.6]{CD16}).

We begin by analyzing the action of the operator $M_\tau$ on the Drury--Arveson space.
In our computations with monomials, we will use familiar multi-index notation. In addition,
we set $\one = (1,1,\ldots,1) \in \bN_0^d$
and define
\begin{equation*}
  \partial \bN_0^d = \{ \alpha \in \bN_0^d: \alpha_j = 0 \text{ for some } j \in \{1,\ldots, d \} \}.
\end{equation*}
Finally, if $\alpha \in \partial \bN_0^d$, let
\begin{equation*}
  X_\alpha = \{ \alpha + k \one: k \in \bN_0 \}
\end{equation*}
and correspondingly, define
\begin{equation*}
  H_\alpha = \ol{ \spa} \{ z^\beta: \beta \in X_\alpha \} \subset H^2_d.
\end{equation*}
Observe that the space $H_0$ is precisely the range of the isometry $V$ of Lemma \ref{lem:H_d_embed}.
Recall that $\tau(z) = d^{d/2} z_1 z_2 \ldots z_d$.

\begin{lem}
  \label{lem:DA_decomp}
  For $d \in \bN$, the space $H^2_d$ admits an orthogonal decomposition
  \begin{equation*}
    H^2_d = \bigoplus_{\alpha \in \partial \bN_0^d} H_\alpha.
  \end{equation*}
  For each $\alpha \in \partial \bN_0^d$, the space $H_\alpha$ is reducing for the multiplication operator $M_\tau$,
  and the operator $M_\tau \big|_{H_\alpha}$ is unitarily equivalent to a unilateral weighted shift with positive weight sequence
  $(w_{k,\alpha})_{k=0}^\infty$, where
  \begin{equation*}
    w_{k, \alpha}^2 = d^d \frac{ \prod_{j=1}^d (\alpha_j+k+1)}{\prod_{j=1}^d ( |\alpha| + k d + j)}.
  \end{equation*}
\end{lem}

\begin{proof}
  To prove the first assertion, it suffices to show that $\bN_0^d$ is the disjoint union
  \begin{equation*}
    \bN_0^d = \bigcup_{ \alpha \in \partial \bN_0^d} X_\alpha.
  \end{equation*}
  To see this, let $\beta \in \bN_0^d$ and let $k = \min (\beta_1,\ldots,\beta_d)$. Then $\beta - k \one \in \partial \bN_0^d$, so $\beta \in X_\alpha$, where $\alpha = \beta - k \one$. To see that the union is disjoint,
  observe that if $\beta = \alpha + k \one$ with $\alpha \in \partial \bN_0^d$,
  then $k = \min (\beta_1,\ldots,\beta_d)$ since $\alpha \in \partial \bN_0^d$,
  hence $\alpha$ is uniquely determined by $\beta$.

  It is clear that each $H_\alpha$ is invariant and hence reducing for $M_\tau$. Moreover,
  with respect to the orthonormal basis $(z^{\alpha + k \one} / ||z^{\alpha + k \one}||)_{k=0}^\infty$ of $H_\alpha$,
  the operator $M_\tau$ is a weighted shift with positive weights $w_{k,\alpha}$ given by
  \begin{align*}
    w_{k,\alpha}^2 = d^d \frac{ ||z^{\alpha + (k+1) \one}||^2}{ ||z^{\alpha + k \one}||^2}
    &= d^d \frac{ (\alpha + (k+1) \one)! \, | \alpha + k \one|!}{(\alpha + k \one)! \, |(\alpha + (k+1) \one)|!} \\
      &= d^d \frac{ \prod_{j=1}^d (\alpha_j+k+1)}{\prod_{j=1}^d ( |\alpha| + k d + j)},
  \end{align*}
  as asserted.
\end{proof}

To prove Proposition \ref{prop:mult_embed}, we need to control the norm of $M_\tau$, and more generally
of operators of the form $M_{\varphi \circ \tau}$, on the spaces $H_\alpha$.
For this purpose, we will use Kacnelson's result \cite{Kacnelson72} is a similar way as we did
in the proof of Lemma \ref{lem:first_proj}. More precisely, we will apply the following
variant of Kacnelson's result.

\begin{cor}
  \label{cor:weighted_shift_domination}
  Let $\cH$ be a Hilbert space with an orthonormal basis $(e_n)_{n=0}^\infty$
  and let $S_1$ and $S_2$ be two weighted shifts given by
  \begin{equation*}
    S_1 e_n = \alpha_n e_{n+1} \quad \text{ and } \quad S_2 e_n = \beta_n e_{n+1},
  \end{equation*}
  where $\alpha_n,\beta_n > 0$ for $n \in \bN_0$.
  If $\beta_n \le \alpha_n$ for all $n \in \bN_0$, then the map
  \begin{equation*} 
    p(S_1) \mapsto p(S_2) \quad (p \in \bC[z])
  \end{equation*}
  is a unital complete contraction.
\end{cor}

\begin{proof}
  Let $\cH_1$ be the space of formal power series in the variable $z$ with orthogonal
  basis $(z^n)_{n=0}^\infty$ and norm defined by
  \begin{equation*}
    \|z^n\|_{\cH_1} = \prod_{j=0}^{n-1} \alpha_j,
  \end{equation*}
  where the empty product is understood as $1$.
  Similarly, $\cH_2$ has a norm defined by
  \begin{equation*}
    \|z^n\|_{\cH_2} = \prod_{j=0}^{n-1} \beta_j.
  \end{equation*}
  Then $S_i$ is unitarily equivalent to the operator of multiplication by $z$ on $\cH_i$ for $i=1,2$,
  and these operators are lower triangular with respect to the orthogonal bases given by the monomials.
  Since $\beta_n \le \alpha_n$ for all $n \in \bN_0$, the sequence $\|z^n\|_{\cH_1} / \|z^n\|_{\cH_2}$
  is non-decreasing, so the result follows from Lemma \ref{lem:kacnelson_triangular}.
\end{proof}

To apply Corollary \ref{cor:weighted_shift_domination}, we require the following combinatorial inequality.
\begin{lem}
  \label{lem:weight_inequality}
  Let $w_{k,\alpha}$ denote the positive weights of Lemma \ref{lem:DA_decomp}, that is,
  \begin{equation*}
    w_{k, \alpha}^2 = d^d \frac{ \prod_{j=1}^d (\alpha_j+k+1)}{\prod_{j=1}^d ( |\alpha| + k d + j)}.
  \end{equation*}
  Then the weights $w_{k,\alpha}$ satisfy the inequalities
  \begin{equation*}
    w_{k, \alpha} \le w_{k,0}
  \end{equation*}
  for all $k \in \bN_0$ and all $\alpha \in \bN_0^d$.
\end{lem}

\begin{proof}
  By the inequality of arithmetic and geometric means,
  \begin{equation*}
    w_{k,\alpha}^2 \le d^d
    \frac{ (\tfrac{|\alpha|}{d} + k+ 1)^{d} }{\prod_{j=1}^d ( |\alpha| + k d + j)},
  \end{equation*}
  so it suffices to show that for all real numbers $r \ge 0$ and all $k \in \bN_0$, the inequality
  \begin{equation*}
    \frac{ (\tfrac{r}{d} + k+ 1)^{d} }{\prod_{j=1}^d ( r + k d + j)}
    \le \frac{(k+1)^{d}}{\prod_{j=1}^d (k d + j)}
  \end{equation*}
  holds. By taking logarithms, we see that this last statement is equivalent to the assertion that for
  any $k \in \bN_0$, the function $f:[0,\infty) \to \bR$ defined by
  \begin{equation*}
    f(r) = d \log( \tfrac{r}{d} + k + 1) - \sum_{j=1}^d \log(r + k d + j),
  \end{equation*}
  has a global maximum at $r=0$.

  To see this, we compute the derivative
  \begin{equation*}
    f'(r) = \frac{1}{\tfrac{r}{d} + k + 1} - \sum_{j=1}^d \frac{1}{r + d k + j}
    =  \frac{1}{\tfrac{r}{d} + k + 1} - \frac{1}{d} \sum_{j=1}^d \frac{1}{\tfrac{r}{d} + k + \tfrac{j}{d}}.
  \end{equation*}
  Since
  \begin{equation*}
    \frac{r}{d} + k + \frac{j}{d} \le \frac{r}{d}+k+1
  \end{equation*}
  for all $j \in \{1,\ldots,d \}$, we deduce from the formula for the derivative of $f$ that $f'(r) \le 0$
  for all $r \in [0,\infty)$. In particular, $f(r) \le f(0)$ for all $r \in [0,\infty)$, which finishes the proof.
\end{proof}

We are now ready to prove Proposition \ref{prop:mult_embed}.

\begin{proof}[Proof of Proposition \ref{prop:mult_embed}]
  We wish to show that the map
  \begin{equation*}
    \Phi: \Mult(\cH_d) \to \Mult(H^2_d), \quad \varphi \mapsto \varphi \circ \tau,
  \end{equation*}
  is a complete isometry.
  Since $f \mapsto f \circ \tau$ is an isometry from $\cH_d$ into $H^2_d$ by part (a) of Lemma \ref{lem:H_d_embed},
  it suffices to show that $\Phi$ is a complete contraction (and in particular maps into $\Mult(H^2_d)$).

  To this end, let $p = [p_{i j}]$ be an $n \times n$-matrix of polynomials.
  Lemma \ref{lem:DA_decomp} shows that $[p_{i j} (M_\tau)] \in \cB((H^2_d)^n)$ is the direct sum
  of the operators
  \begin{equation*}
    [p_{i j} (M_\tau \big|_{H_\alpha})] \in \cB(H_\alpha^n)
  \end{equation*}
  for $\alpha \in \partial \bN_0^d$,
  and that each $M_\tau \big|_{H_\alpha}$ is unitarily equivalent to a weighted shift with
  weight sequence $(w_{k, \alpha})$. In this setting, Lemma \ref{lem:weight_inequality} and Corollary \ref{cor:weighted_shift_domination}
  show that
  \begin{equation*}
    || [p_{i j} (M_\tau \big|_{H_\alpha})] || \le || [p_{i j} (M_\tau \big|_{H_0}) ] ||
    = || [p_{i j} (M_z) ||_{\cB(\cH_d)^d},
  \end{equation*}
  where the last equality follows from part (a) of Lemma \ref{lem:H_d_embed}. Thus,
  \begin{equation*}
  || [{p_{i j} \circ \tau}] ||_{\Mult(H^2_d)^n} \le || [{p_{i j}}] ||_{\Mult(\cH_d)^n}.
  \end{equation*}

  The preceding paragraph shows that $\Phi$ is a unital complete contraction on the subspace
  of all polynomials in $\Mult(\cH_d)$. It is well known that circular invariance of $\cH_d$
  implies that the unit ball of $\Mult(\cH_d)$ contains a WOT-dense subset consisting of polynomials, so
  a straightforward limiting argument finishes the proof.
\end{proof}

Our first application of Proposition \ref{prop:mult_embed} concerns subhomogeneity of $\Mult(H^2_d)$.

\begin{cor}
  \label{cor:da_not_subhom}
  If $d \ge 2$, then $\Mult(H^2_d)$ is not topologically subhomogeneous.
  In particular, there do not exist a constant $C > 0$ and $n \in \bN$ such that
  \begin{equation*}
    ||\varphi||_{\Mult(H^2_d)} \le C ||\varphi||_{\Mult(H^2_d),n}
  \end{equation*}
  for all $\varphi \in \Mult(H^2_d)$.
\end{cor}

\begin{proof}
  By Corollary \ref{cor:not_subhom}, the multiplier algebra of the weighted Dirichlet
  space $\cD_{1/2}$ is not topologically subhomogeneous. By Proposition \ref{prop:mult_embed} and part (c)
  of Lemma \ref{lem:H_d_embed}, $\Mult(\cD_{1/2})$ is (completely isometrically) isomorphic
  to a subalgebra of $\Mult(H^2_2)$, hence $\Mult(H^2_2)$ is not topologically subhomogeneous. The case of
  general $d$ now follows from Lemma \ref{lem:first_proj} (or from Remark \ref{rem:first_proj_cnp}).
  The additional statement is once again a consequence of the fact that the $n$-point multiplier norm
  can be computed using representations of dimension at most $n$.
\end{proof}

\section{\texorpdfstring{Applications to multipliers of $H^2_d$}{Applications to multipliers of H2d}}

\label{sec:embedding_applications}

We will use Proposition \ref{prop:mult_embed} to establish two more results regarding multipliers
of the Drury--Arveson space.

Our first application concerns the Sarason function, which was studied in \cite{AHM+17c}. If $\cH$
is a complete Nevanlinna--Pick space whose kernel $K$ is normalized at a point and if $f \in \cH$, then the Sarason function
of $f$ is defined by
\begin{equation*}
  V_f(z) = 2 \langle f, K(\cdot,z) f \rangle - \|f\|^2.
\end{equation*}
If $\cH = H^2$, then $\Re V_f$ is the Poisson integral of $|f|^2$. In particular, $f \in H^\infty$ if and only
if $\Re V_f$ is bounded.

In Theorem 4.5 of \cite{AHM+17c}, it is shown that
for a class of complete Pick spaces
including standard weighted Dirichlet spaces on the disc and the Drury--Arveson space,
boundedness of $\Re V_f$ implies that $f \in \Mult(\cH)$.
Proposition 4.8 of \cite{AHM+17c} shows that the converse
fails for the standard weighted Dirichlet spaces $\cD_{a}$ for $0 < a < 1$, that is,
there are multipliers of $\cD_a$ whose Sarason function has unbounded real part.

With the help of Proposition \ref{prop:mult_embed}, we can establish such a result
for the Drury--Arveson space as well. This answers a question of J\"org Eschmeier \cite{EschmeierPC}.
This result was also obtained with a different proof in the recent paper \cite{FX20} by Fang and Xia.

\begin{prop}
  \label{prop:da_sarason}
  For $d \ge 2$, there exist $\varphi \in \Mult(H^2_d)$ such that $\Re V_\varphi$ is unbounded.
\end{prop}

\begin{proof}
It suffices to construct such a multiplier $\varphi$ for $d = 2$, as the trivial extension
of $\varphi$ to $\bB_d$ for $d \ge 2$ (i.e. the function $\varphi \circ P$, where
$P$ is the projection on the first coordinates) will be a multiplier of $H^2_d$ (by Lemma \ref{lem:first_proj})
whose Sarason function agrees with that of $\varphi$ on $\bB_2$.

  By Proposition 4.8 of \cite{AHM+17c}, there exists $u \in \Mult(\cD_{1/2})$ such that $\Re V_u$ is
  unbounded in  $\bD$. Let $\tau: \bB_2 \to \bD$ be defined as in Section \ref{sec:embedding} and let
  $\varphi = u \circ \tau$. Then $\varphi \in \Mult(H^2_2)$ by Proposition \ref{prop:mult_embed}. It remains
  to show that $\Re V_\varphi$ is unbounded in $\bB_2$.

  To this end,
  let $M = \ol{\spa}\{ (z_1 z_2)^n: n \in \bN_0 \} \subset H^2_2$, let $L$ denote
  the reproducing kernel of $M$ and
  let $K$ be the reproducing kernel of $\cD_{1/2}$.
  It follows from part (c) of Lemma \ref{lem:H_d_embed} (or by direct computation) that
  \begin{equation*}
    K(\tau(z),\tau(w)) = L(z,w) \quad (z,w \in \bB_2).
  \end{equation*}
  Indeed, since $V$ is given by composition with $\tau$, we have $V^* L(\cdot,w) = K(\cdot, \tau(w))$,
  and since $V: \cD_{1/2} \to M$ is a unitary, $L(\cdot,w) = V V^* L(\cdot,w) = K(\cdot, \tau(w)) \circ \tau$.

  Since $\varphi \in M$, it is elementary to check that if $h \in M^\bot$, then $\varphi \cdot h \in M^\bot$.
  Using the fact that $L(\cdot,w) = P_M S(\cdot,w)$, where $S$ denotes the reproducing
  kernel of $H^2_d$, we therefore find that for $w \in \bB_d$,
  \begin{equation*}
    \langle \varphi, S(\cdot,w) \varphi \rangle = \langle \varphi, L(\cdot,w) \varphi  \rangle
    = \langle \varphi, (K(\cdot,\tau(w)) \circ \tau) \varphi \rangle = \langle u, K(\cdot,\tau(w))  u \rangle_{\cD_{1/2}},
  \end{equation*}
  where the last equality follows from part (a) of Lemma \ref{lem:H_d_embed} and the definition of $\varphi = u \circ \tau$.
  Therefore, $V_\varphi = V_u \circ \tau$.
  Since $\tau$ is surjective and $\Re V_u$ is unbounded, $\Re V_\varphi$ is unbounded as well.
\end{proof}

Our second application concerns subspaces of $H^2_d$ that entirely consist of multipliers.
It is a special case of a theorem of Grothendieck \cite{Grothendieck54}
that every closed infinite dimensional subspace of
$L^2(\bT)$ contains functions that are not essentially bounded; see \cite[Theorem 5.2]{Rudin91} for a fairly
elementary proof.
Thus, no infinite dimensional closed subspace of $H^2$ consists
entirely of multipliers of $H^2$. J\"org Eschmeier \cite{EschmeierPC} asked if there exist closed infinite
dimensional subspaces of $H^2_d$ that entirely consist of multipliers and that are graded in the sense that they are spanned by homogeneous polynomials. We use Proposition \ref{prop:mult_embed}
to provide a positive answer if $d \ge 4$.

\begin{cor}
  If $d \ge 4$, then the space
  \begin{equation*}
    H_0 = \ol{\spa} \{ (z_1 z_2 \cdots z_d)^n: n \in \bN_0 \} \subset H^2_d
  \end{equation*}
  is contained in $\Mult(H^2_d)$.
\end{cor}

\begin{proof}
  Part (a) of Lemma \ref{lem:H_d_embed} shows that $H_0 = \{f \circ \tau: f \in \cH_d \}$.
  But $\cH_d = \Mult(\cH_d)$ by part (e) of Lemma \ref{lem:H_d_embed}, so that the result
  follows from Proposition \ref{prop:mult_embed}.
\end{proof}

\section{Spaces between the Drury--Arveson space and the Hardy space}
\label{sec:spaces_between}

Recall that in the scale of spaces $\cD_a(\bB_d)$, the values
$a=1$ and $a=d$ correspond to the Drury--Arveson space and to the Hardy space, respectively.
We have already seen that $\Mult(\cD_a(\bB_d))$ is not topologically subhomogeneous
in the case $0 \le a < 1$ and also in the case $a=1$ and $d \ge 2$; see Corollary \ref{cor:ball_not_subhom}
and Corollary \ref{cor:da_not_subhom}.
On the other
hand, if $a \ge d$, then $\Mult(\cD_a(\bB_d)) = H^\infty(\bB_d)$, which is (even completely isometrically) $1$-subhomogeneous.
In this section, we study topological subhomogeneity of $\Mult(\cD_a(\bB_d))$ for $1 < a < d$.

In fact, it will occasionally be more convenient to work with an equivalent norm on $\cD_a(\bB_d)$. To this end,
let $s \in \bR$ and let $\cH_s(\bB_d)$ be the reproducing kernel Hilbert space on $\bB_d$ with
kernel
\begin{equation*}
  \sum_{n=0}^\infty (n+1)^{s} \langle z,w \rangle^n.
\end{equation*}
Equivalently,
\begin{equation*}
  \cH_s(\mathbb{B}_d) =
  \Big\{ f = \sum_{\alpha \in \mathbb{N}_0^d} \widehat{f}(\alpha) z^\alpha \in \mathcal{O}(\mathbb{B}_d): \|f\|^2
  = \sum_{\alpha \in \mathbb{N}_0^d} (|\alpha|+1)^{-s} \frac{\alpha!}{|\alpha|!} |\widehat{f}(\alpha)|^2 < \infty \Big\}.
\end{equation*}
We simply write $\cH_s = \cH_s(\bD)$.
It is well known that if $s = a - 1 \ge -1$, then $\cD_a(\bB_d) = \cH_s(\bB_d)$ with equivalence of norms.
Indeed, the weights in the description of $\mathcal{D}_a(\mathbb{B}_d)$ are comparable
to those in the description of $\mathcal{H}_s(\mathbb{B}_d)$, because for $a > 0$, we have
\begin{equation*}
  \frac{1}{\Gamma(a+n)} \approx \frac{(n+1)^{1-a}}{n!};
\end{equation*}
see for instance \cite{Wendel48}.

As in the case of the Drury--Arveson space, it is possible to embed weighted Dirichlet spaces on $\bD$
into $\cD_a(\bB_d)$ for certain values of $a$.
As in Section \ref{sec:embedding}, we let
\begin{equation*}
  \tau: \ol{\bB_d} \to \ol{\bD}, \quad z \mapsto d^{d/2} z_1 z_2 \ldots z_d.
\end{equation*}

\begin{lem}
  \label{lem:Besov_embed_space}
  Let $s \in \bR$.
  The map
  \begin{equation*}
  \cH_{s-(d-1)/2} \to \cH_s(\bB_d), \quad f \mapsto f \circ \tau,
  \end{equation*}
  is bounded and bounded below.
\end{lem}

\begin{proof}
  The monomials form an orthogonal basis of $\cH_{s-(d-1)/2}$, and powers of $\tau$ are orthogonal
  in $\cH_s(\bB_d)$. Moreover, by Stirling's formula,
  we have
  \begin{align*}
    \|\tau^n\|_{\cH_s(\bB_d)}^2 &= d^{n d} \frac{(n!)^d}{(n d)!} (nd + 1)^{-s} \approx
    (n+1)^{(d-1)/2} ( n d + 1)^{-s} \approx (n+1)^{ (d-1)/2 - s} \\
                                &= \|z^n\|^2_{\cH_{s - (d-1)/2}},
  \end{align*}
  so the result follows.
\end{proof}

Our goal is to show that in analogy with Proposition \ref{prop:mult_embed},
the map $\varphi \mapsto \varphi \circ \tau$ also yields an embedding on the level of multiplier algebras.
Using notation as in Section \ref{sec:embedding}, we define for $\alpha \in \partial \bN_0^d$
a space
\begin{equation*}
  H_\alpha^{(s)} = 
  \ol{ \spa} \{ z^\beta: \beta \in X_\alpha \} \subset \cH_s(\bB_d).
\end{equation*}
Then the following generalization of Lemma \ref{lem:DA_decomp} is proved in the same way.

\begin{lem}
  \label{lem:H_s_decomp}
  Let $d \in \bN$ and $s \in \bR$. Then $\cH_s(\bB_d)$ admits an orthogonal decomposition
  \begin{equation*}
    \cH_s(\bB_d) = \bigoplus_{\alpha \in \partial \bN_0^d} H_\alpha^{(s)}
  \end{equation*}
  For each $\alpha \in \partial \bN_0^d$, the space $H_\alpha^{(s)}$ is reducing for the multiplication operator $M_\tau$,
  and the operator $M_\tau \big|_{H_\alpha^{(a)}}$ is unitarily equivalent to a unilateral weighted shift with positive weight sequence
  $(w_{k,\alpha})_{k=0}^\infty$, where
  \begin{equation*}
    \pushQED{\qed}
    w_{k, \alpha}^2 = d^d \frac{ \prod_{j=1}^d (\alpha_j+k+1)}{\prod_{j=1}^d ( |\alpha| + k d + j)}
    \Big( 1 + \frac{d}{|\alpha| + k d + 1} \Big)^{-s}. \qedhere
    \popQED
  \end{equation*}
\end{lem}

For $r \in \bN_0$ and $k \in \bN_0$, we define positive weights by
\begin{equation*}
  v_{k,r}^2 = 
  d^d \frac{(\frac{r}{d} + k + 1)^d}{\prod_{j=1}^d (r + k d + j)}
    \Big( 1 + \frac{d}{r + k d + 1} \Big)^{-s}
\end{equation*}
and
\begin{equation*}
  u_k^2 =
  d^d \frac{(k + 1)^d}{\prod_{j=1}^d (k d + j)}
    \Big( 1 + \frac{d}{d + k d + 1} \Big)^{-s}
\end{equation*}
Then the following inequalities hold.

\begin{lem}
  \label{lem:weights_inequalities}
  Let $w_{k,\alpha}$ denote the weights of Lemma \ref{lem:H_s_decomp}.
  Let $r,k \in \bN_0$ and $\alpha \in \bN_0^d$.
  \begin{enumerate}[label=\normalfont{(\alph*)}]
    \item $w_{k,\alpha} \le v_{k, |\alpha|}$.
    \item If $r = l d + t$ with $l,t \in \bN_0$, then $v_{k,r} = v_{k+l,t}$.
    \item Assume $s \le 0$. Then $v_{k,r} \le v_{k,0}$
      and thus $w_{k,\alpha} \le w_{k,0}$.
    \item Assume $s \ge 0$. If $0 \le r \le d$, then $v_{k,r} \le u_k$.
  \end{enumerate}
\end{lem}

\begin{proof}
  (a) follows from the inequality of arithmetic and geometric means.

  (b) is obvious.
  
  (c) and (d) We saw in the proof of Lemma \ref{lem:weight_inequality} that
  \begin{equation*}
  \frac{(\frac{r}{d} + k + 1)^d}{\prod_{j=1}^d (r + k d + j)} \le
  \frac{(k+1)^d}{ \prod_{j=1}^d (kd + j)}.
  \end{equation*}
  From this inequality and from the fact that the quantity
  \begin{equation*}
    \Big( 1 + \frac{d}{r+k d + 1} \Big)^{-s} 
  \end{equation*}
  is non-increasing in $r$ if $s \le 0$ and non-decreasing in $r$ if $s \ge 0$,
  the inequalities involving $v_{k,r}$ in (c) and (d) follow. The inequality about $w_{k,\alpha}$
  in (c) then follows from the inequality about $v_{k,r}$ and part (a).
\end{proof}

We are now ready to prove the desired embedding result for multiplier algebras, which
extends Proposition \ref{prop:mult_embed}.

\begin{prop}
  \label{prop:besov_mult_embed}
  Let $s \in \bR$. Then the map
  \begin{equation*}
    \Mult(\cH_{s-(d-1)/2}) \to \Mult(\cH_s(\bB_d)), \quad \varphi \mapsto \varphi \circ \tau,
  \end{equation*}
  is a completely bounded isomorphism onto its image.

  In particular, if $a \ge \frac{d-1}{2}$, then
  \begin{equation*}
    \Mult(\cD_{a-(d-1)/2}) \to \Mult(\cD_a(\bB_d)), \quad \varphi \mapsto \varphi \circ \tau,
  \end{equation*}
  is a completely bounded isomorphism onto its image.
\end{prop}

\begin{proof}
  The second statement follows from the first statement and the equality
  $\cD_a(\bB_d) = \cH_{a-1}(\bB_d)$ with equivalence of norms.
  To prove the first statement,
  in light of Lemma \ref{lem:Besov_embed_space}, it suffices to show that the map is completely bounded.
  Moreover, as in the proof of Proposition \ref{prop:mult_embed},
  it suffices to show that the map is completely bounded on the space of all polynomials by an approximation argument.
  Let $M_z$ denote the operator of multiplication by $z$ on $\cH_{s-(d-1)/2}$.
  Lemma \ref{lem:H_s_decomp} further implies that it suffices to prove that
  there exists a constant $C > 0$ so that for all $\alpha \in \partial \bN_0^d$, the mapping
  \begin{equation}
    \label{eqn:besov_embed_proof}
    p(M_z) \mapsto p(M_\tau \big|_{H_\alpha^{(s)}}) \quad (p \in \bC[z])
  \end{equation}
  has completely bounded norm at most $C$.
  
  Recall from Lemma \ref{lem:H_s_decomp} that the operator $M_{\tau} \big|_{H_\alpha^{(s)}}$
  is a weighted shift with weight sequence $(w_{k,\alpha})_{k=0}^\infty$.
  We introduce the following notation. Given
  two weighted shifts $S$ and $T$ with positive weights, we write
  $S \prec T$ if the map
  \begin{equation*}
    p(T) \mapsto p(S), \quad (p \in \bC[z])
  \end{equation*}
  is completely contractive.
  
  Suppose first that $s \le 0$. Then part (c) of Lemma \ref{lem:weights_inequalities} and Corollary
  \ref{cor:weighted_shift_domination} imply that
  \begin{equation*}
    M_{\tau} \big|_{H_\alpha^{(s)}} \prec
    M_{\tau} \big|_{H_0^{(s)}}
  \end{equation*}
  for all $\alpha \in \partial \bN_0^d$. Moreover, Lemma \ref{lem:Besov_embed_space} implies that
  the map
  \begin{equation*}
    p(M_z) \mapsto p(M_{\tau} \big|_{H_0^{(s)}})
  \end{equation*}
  is completely bounded, hence the mapping in Equation \eqref{eqn:besov_embed_proof} is completely bounded
  with completely bounded norm independent of $\alpha \in \partial \bN_0^d$, which finishes the proof in the case $s \le 0$.
  
  Suppose now that $s \ge 0$.
  Let $S_r$ denote the weighted shift with weights $(v_{k,r})_{k=0}^\infty$
  and let $S$ denote the weighted shift with weights $(u_k)_{k=0}^\infty$,
  defined before Lemma \ref{lem:weights_inequalities}.
  Let $\alpha \in \bN_0^d$.
  Then by part (a) of Lemma \ref{lem:weights_inequalities} and by Corollary \ref{cor:weighted_shift_domination},
  \begin{equation*}
    M_{\tau} \big|_{H_\alpha^{(s)}} \prec S_{|\alpha|}.
  \end{equation*}
  Write $|\alpha| = l d + t$ with $l,t \in \bN_0$ and $0 \le t \le d$. Part (b) of Lemma \ref{lem:weights_inequalities}
  shows that $v_{k,|\alpha|} = v_{k+l,t}$, so that $S_{|\alpha|}$
  is unitarily equivalent to a restriction of $S_t$ to an invariant subspace and hence
  \begin{equation*}
    S_{|\alpha|} \prec S_{t}.
  \end{equation*}
  Finally, part (d) of Lemma \ref{lem:weights_inequalities} and Corollary \ref{cor:weighted_shift_domination}
  imply that 
  \begin{equation*}
    S_t \prec S,
  \end{equation*}
  so combining the last three relations, we see that
  \begin{equation*}
    M_{\tau} \big|_{H_\alpha^{(s)}} \prec S.
  \end{equation*}
  
  We finish the proof by showing that the map
  \begin{equation*}
    p(M_z) \mapsto p(S)
  \end{equation*}
  is completely bounded, where $M_z$ continues to denote the operator of multiplication by $z$
  on $\cH_{s-(d-1)/2}$.
  To this end, let $\cK$ be the space of power series with orthogonal basis $(z^n)_{n=0}^\infty$
  and norm
  \begin{equation*}
    \|z^n\|^2_{\cK} = d^{n d} \frac{(n!)^d}{(n d)!} \Big(1 + \frac{n d}{d+1} \Big)^{-s}.
  \end{equation*}
  Since
  \begin{equation*}
    \frac{\|z^{k+1}\|^2_{\cK}}{\|z^k\|^2_{\cK}} = u_k^2,
  \end{equation*}
  we see that $S$ is unitarily equivalent to the operator of multiplication by $z$ on $\cK$.
  Moreover, by Stirling's formula,
  \begin{equation*}
    \|z^n\|_{\cK}^2 \approx (n+1)^{(d-1)/2} (n+1)^{-s} = \|z^n\|^2_{\cH_{s-(d-1)/2}},
  \end{equation*}
  so that $\cK = \cH_{s-(d-1)/2}$ with equivalence of norms,
  hence $S$ and $M_z$ are similar.
\end{proof}

We can now show that multiplier algebras of some spaces
between the Drury--Arveson space and the Hardy space are not subhomogeneous.

\begin{cor}
  \label{cor:between_spaces_subhom}
  Let $d \in \bN$ and let $0 \le a < \frac{d+1}{2}$. Then $\Mult(\cD_a(\bB_d))$
  is not topologically subhomogeneous.
\end{cor}

\begin{proof}
  We combine two arguments. Firstly,
  if $\frac{d-1}{2} \le a < \frac{d+1}{2}$, then Proposition \ref{prop:besov_mult_embed} shows
  that $\Mult(\cD_{a}(\bB_d))$ contains a subalgebra isomorphic to $\Mult(\cD_{a-(d-1)/2})$.
  Since $0 \le a- \frac{d-1}{2} < 1$, the algebra $\Mult(\cD_{a - (d-1)/2})$ is not
  topologically subhomogeneous by Corollary \ref{cor:not_subhom} and Proposition \ref{prop:dirichlet_subhom}.
  Thus, $\Mult(\cD_a(\bB_d))$ is not topologically subhomogeneous for $\frac{d-1}{2} \le a < \frac{d+1}{2}$.

  Secondly, if $\Mult(\cD_a(\bB_{d}))$ is not topologically subhomogeneous, then Lemma \ref{lem:first_proj}
  implies that $\Mult(\cD_a(\bB_{d+1}))$ is not topologically subhomogeneous either. Combining
  these two statements, the result follows by an obvious induction on $d$.
\end{proof}

Corollary \ref{cor:between_spaces_subhom} leaves open the question of subhomogeneity
of $\Mult(\cD_a(\bB_d))$
in the range $\frac{d+1}{2} \le a < d$. We now show that in this range,
the multiplier norm is not comparable to the supremum norm, from which we deduce that $\Mult(\cD_a(\bB_d))$
is at least not topologically $1$-subhomogeneous. To this end, we will use inner functions on the unit ball.
Recall that a bounded holomorphic function $f: \bB_d \to \bC$ is said to be inner
if $\lim_{r \to 1} |f(r z)| = 1$ for almost every $z \in \partial \bB_d$.

\begin{lem}
  \label{lem:inner}
  Let $0 \le a < d$ and let $f$ be a non-constant inner function on $\bB_d$.
  Then $\lim_{n \to \infty} \|f^n\|_{\cD_a(\bB_d)} = \infty$.
\end{lem}

\begin{proof}
  If $0 \le b < a$, then $D_b(\mathbb{B}_d)$ is continuously contained in $\mathcal{D}_a(\mathbb{B}_d)$,
  which is easily seen by passing to the equivalent norm in $\mathcal{H}_s(\mathbb{B}_d)$ mentioned at the beginning of this section.
  Therefore, it suffices to prove the result for $a$ close to $d$, say $a = d - \delta$, where $0 < \delta < 1$.
  It is well known that another equivalent norm on $\cD_a(\bB_d)$ is given by
  \begin{equation*}
    \|h\|^2 = |h(0)|^2 + \int_{\bB_d} |R h(z)|^2 (1 - |z|^2)^{1 -\delta} \, dV(z),
  \end{equation*}
  where
  $R$ denotes the radial derivative and $V$ is the Lebesgue measure on $\bB_d$; see
  \cite[Theorem 41]{ZZ08}. Moreover, \cite[Theorem 5.3]{Stoll13} (choosing $q=2$ and $p= \frac{d}{d+2-\delta}$) implies that for inner functions $g$,
  \begin{equation*}
   \int_{\bB_d} |R g(z)|^2 (1 - |z|^2)^{1 -\delta} \, dV(z)
    \approx \int_{\bB_d} (1 - |g(z)|^2)^2 (1 - |z|^2)^{-\delta - 1} \, dV(z),
  \end{equation*}
  where the implied constants do not depend
  on the inner function $g$. Choosing $g=f^n$ above, we find that
  \begin{equation*}
    \|f^n\|_{\cD_a(\bB_d)}^2
    \gtrsim \int_{\bB_d} (1 - |f^n(z)|^2)^2 (1 - |z|^2)^{-\delta -1} \, dV(z).
  \end{equation*}
  Since $f$ is not constant, $f^n$ converges to $0$ pointwise on $\bB_d$, so the monotone
  convergence theorem implies that the last integral tends to
  \begin{equation*}
    \int_{\bB_d} (1 - |z|^2)^{-\delta - 1} \, dV(z) = \infty,
  \end{equation*}
  which finishes the proof.
\end{proof}

We are now ready to prove the announced result about $1$-subhomogeneity.

\begin{prop}
  \label{prop:not-1-subhom}
  Let $0 \le a < d$. Then the multiplier norm on $\cD_a(\bB_d)$ is not
  comparable to the supremum norm on $\bB_d$. In fact, $\Mult(\cD_a(\bB_d))$
  is not topologically $1$-subhomogeneous.
\end{prop}

\begin{proof}
  If $0 \le a < 1$, then the result is a special case of Corollary \ref{cor:ball_not_subhom},
  so we may assume that $1 \le a < d$.
  We first show that there does not exist a constant $C > 0$ so that
  \begin{equation*}
    \|p\|_{\Mult(\cD_a(\bB_d))} \le C \|p\|_\infty
  \end{equation*}
  for all polynomials $p$. Suppose otherwise. Since every function $f \in H^\infty(\bB_d)$
  is a pointwise limit of polynomials $p_n$ with $\|p_n\|_\infty \le \|f\|$
  (for instance the Fej\'er means of $f$), it follows that every function $f \in H^\infty(\bB_d)$
  is a multiplier of $\cD_a(\bB_d)$ with
  \begin{align*}\|f\|_{\Mult(\cD_a(\bB_d))} \le C \|f\|_\infty.\end{align*}
  On the other hand, if $f$ is a non-constant inner function, which exists by a theorem
  of Aleksandrov \cite{Aleksandrov82}, then
  $\|f^n\|_\infty \le 1$ for all $n \in \bN$, but Lemma \ref{lem:inner} implies that
  \begin{equation*}
    \|f^n\|_{\Mult(\cD_a(\bB_d))} \ge
    \|f^n\|_{\cD_a(\bB_d)} \xrightarrow{n \to \infty} \infty.
  \end{equation*}
  This contradiction proves the claim and in particular the first part of the proposition.

  Next, let $A$ denote the norm closure of the polynomials inside of $\Mult(\cD_a(\bB_d))$.
  We will show that $A$ is not topologically $1$-subhomogeneous, which will
  finish the proof.
  To this end, we determine the unital
  homomorphisms $\chi: A \to \bC$, i.e. the characters of $A$.
  We claim that if $\chi$ is a character, then
  then there exists $\lambda \in \ol{\bB_d}$ so that $\chi(f) = f(\lambda)$
  for all $f \in A$. Indeed, given a character $\chi$, define
  $\lambda = (\chi(z_1),\ldots,\chi(z_d))$.
  If $K_a$ denotes the reproducing kernel of $\cD_a(\bB_d)$, then
  \begin{equation*}
    (1 - \langle z, w\rangle) K_a(z,w) = \frac{1}{(1 - \langle z,w \rangle)^{a-1}}
  \end{equation*}
  is positive, hence the tuple $(M_{z_1},\ldots,M_{z_d})$ forms a row contraction
  on $\cD_a(\bB_d)$.
  Since characters are automatically contractive, and contractive functionals are completely contractive,
  it follows that $\lambda \in \ol{\bB_d}$. (This can also be seen by computing
  the spectrum of the tuple $M_z$.) Clearly, $\chi(f) = f(\lambda)$ for all
  polynomials $f$, and hence for all $f \in A$ by definition of $A$.

  In particular, the description of the characters of $A$ implies that
  for all $f \in A$,
  \begin{equation*}
    \sup \{ |\chi(f)| : \chi \text{ is a character of }  A \} = \|f\|_{\infty}.
  \end{equation*}
  Since the multiplier norm of polynomials is not bounded by a constant times
  the supremum norm by the first part,
  it therefore follows that there does not exist
  a constant $C > 0$ so that
  \begin{equation*}
    \|f\|_{\Mult(\cD_a(\bB_d))} \le C \sup \{ |\chi(f)| : \chi \text{ is a character of }  A \}
  \end{equation*}
  for all $f \in A$. In other words, $A$ is not topologically $1$-subhomogeneous.
\end{proof}

\section{Questions}
\label{sec:questions}

We close this article with a few questions.

\subsection{Small Dirichlet spaces}
It is well known that the scale of spaces $\cD_a$, defined for $a \in [0,\infty)$, can be extended as
follows. For $s \in \bR$, let $\cH_s$ be the reproducing kernel Hilbert space on $\bD$ with
kernel
\begin{equation*}
  \sum_{n=0}^\infty (n+1)^{s} (z \overline{w})^n.
\end{equation*}
If $s = a - 1 \ge -1$, then $\cD_a = \cH_s$ with equivalence of norms,
see the discussion at the beginning of Section \ref{sec:spaces_between}.
Moreover, if  $s \ge 0$, then $\|z\|_{\Mult(\mathcal{H}_s)} \le 1$,
so von Neumann's inequality implies that
$\Mult(\cH_s) = H^\infty$ completely isometrically for $s \ge 0$.
On the other hand, it follows from Theorem \ref{thm:ci_subhom_intro}
that $\Mult(\cH_s)$ is not completely isometrically subhomogeneous for all $s < 0$,
since the multiplier norm of $\cH_s$ it not equivalent to the supremum norm for $s < 0$.
In fact, by Theorem \ref{thm:top_subhom_intro}, the algebra $\Mult(\cH_s)$ is not topologically
subhomogeneous for $-1 \le s < 0$, but the proof does not seem to extend to the range $s < -1$.

\begin{quest}
  Is $\Mult(\cH_s)$ topologically subhomogeneous for $s < -1$?
\end{quest}

It is known that for $s < -1$, we have $\Mult(\cH_s) = \cH_s$ with equivalence of norms
(see Proposition 31 and Example 1 on page 99 in \cite{Shields74}). In particular,
$\Mult(\cH_s)$ is reflexive as a Banach space for $s < -1$. Thus, we ask more generally.

\begin{quest}
  Does there exist a unital, infinite-dimensional operator algebra that is subhomogeneous
and reflexive as a Banach space?
\end{quest}

Notice that the answer to this question is independent of what kind of subhomogeneity
is specified, as reflexivity is stable under topological isomorphisms.

\subsection{Spaces between the Drury--Arveson space and the Hardy space}

Consider the spaces $\cD_a(\bB_d)$ for $d \ge 2$,
and recall that $\cD_1(\bB_d)$ is the Drury--Arveson space $H^2_d$ and that
$\cD_d(\bB_d)$ is the Hardy space $H^2(\bB_d)$.
We saw in Corollary \ref{cor:ci_subhom_ball_dirichlet}
that $\Mult(\cD_a(\bB_d))$ is completely isometrically subhomogeneous if and only if $a \ge d$.
On the other hand, Theorem \ref{thm:top_subhom_intro} implies that $\Mult(\cD_a(\bB_d))$
is not even topologically subhomogeneous if $0 \le a < \frac{d+1}{2}$.
Moreover,
$\Mult(\cD_a(\bB_d))$ is not topologically $1$-subhomogeneous for
$0 \le a < d$ by Proposition \ref{prop:not-1-subhom}.

\begin{quest}
  \label{quest:subhom}
  Is $\Mult(\cD_a(\bB_d))$ topologically subhomogeneous if $\frac{d+1}{2} \le a < d$?
\end{quest}

The embedding of Proposition \ref{prop:besov_mult_embed} does not appear to be useful
for the study of this question. For real numbers $a$ in the range $\frac{d+1}{2} \le a < d$,
Proposition \ref{prop:besov_mult_embed} yields embeddings of multiplier algebras
that coincide with $H^\infty(\bD)$, a $1$-subhomogeneous algebra. Thus, a different argument
seems to be needed to cover the range
$\frac{d+1}{2} \le a < d$.

\subsection{Similarity of multipliers}
Suppose that $\cH$ and $\cK$ are reproducing kernel Hilbert spaces on the same set $X$.
If there exists a multiplier $\theta \in \Mult(\cH,\cK)$ so that
\begin{equation*}
  S: \cH \to \cK, \quad f \mapsto \theta f,
\end{equation*}
is bounded and invertible, then $\Mult(\cH) = \Mult(\cK)$ with equivalence of norms.
More precisely, if $C = \|S\| \|S^{-1}\|$, then
\begin{equation*}
  C^{-1} \|\varphi\|_{\Mult(\cH)} \le \|\varphi\|_{\Mult(\cK)} \le C  \|\varphi\|_{\Mult(\cH)}
\end{equation*}
for all $\varphi \in \Mult(\cH)$ (and these inequalities continue to hold for matrix-valued multipliers).
Thus, a positive answer to the following question would provide another proof of Corollary \ref{cor:many_points}
(with potentially different constants).

\begin{quest}
  Let $a \in (0,1)$ and let $n \in \bN$. Does there exist a constant $A > 0$ so that
  for all finite sets $F \subset \bD$ with $|F| = n$, there exists a multiplier
  $\theta \in \Mult( H^2 |_F, \cD_a |_F)$ with
  \begin{equation*}
    \|\theta\|_{\Mult(H^2 |_F, \cD_a |_F)} \|\theta^{-1}\|_{\Mult(\cD_a |_F, H^2 |_F)} \le A?
  \end{equation*}
\end{quest}

\subsection*{Note added in proof}
The recent preprint \cite{HRS21} shows that $\Mult(\mathcal{D}_a(\mathbb{B}_d))$ is not topologically subhomogeneous for $0 \le a < d$, thus answering Question \ref{quest:subhom}.

\bibliographystyle{amsplain}
\bibliography{literature}

\end{document}